\documentclass[10pt, leqno]{amsart}
\usepackage{amscd}
\usepackage{amsaddr}
\usepackage{amssymb}
\usepackage{amsmath}
\usepackage{stmaryrd}
\usepackage{amsthm}
\usepackage[cmtip,all]{xypic}
\usepackage{pifont}

\usepackage[usenames, dvipsnames]{color}

\theoremstyle{plain}
\newtheorem{Lem}{Lemma}[section]

\newtheorem{Cor}[Lem]{Corollary}
\newtheorem{Thm}[Lem]{Theorem}

{\theoremstyle{definition} 

\newtheorem{Rk}[Lem]{Remark}
\newtheorem{Def}[Lem]{Definition}}

\newcommand{\zig}{\addtocounter{Lem}{1}\tag{\theLem}}

\newcommand{\gspt}{G\text{-}\Sigma\mathrm{Sp}}
\newcommand{\zpspt}{\mathbb{Z}_p^\times\text{-}\Sigma\mathrm{Sp}}

\newcommand{\typicaldot}{{\text{\textbullet}}}

\def\:{\colon}
\DeclareMathOperator*{\colim}{colim}
\DeclareMathOperator*{\holim}{holim}
\DeclareMathOperator*{\holimG}{holim^\mathit{G}}
\DeclareMathOperator*{\lims}{lim^\mathit{s}}

\begin{document}
\title{A construction of some objects in many base cases of an Ausoni-Rognes conjecture}
\author{Daniel G. Davis}
\begin{abstract}
Let $p$ be a prime, $n \geq 1$, $K(n)$ 
the $n$th Morava $K$--theory spectrum, $\mathbb{G}_n$ the extended Morava stabilizer group, and $K(A)$ the algebraic $K$--theory spectrum of a commutative $S$--algebra $A$. For a type $n+1$ complex $V_n$, 
Ausoni and Rognes conjectured that 
(a) the unit map $i_n \: L_{K(n)}(S^0) \to E_n$ 
from the $K(n)$--local sphere to the Lubin-Tate spectrum 
induces a map 
\[K(L_{K(n)}(S^0)) \wedge v_{n+1}^{-1}V_n \to (K(E_n))^{h\mathbb{G}_n} 
\wedge v_{n+1}^{-1}V_n\] that is a weak equivalence, where (b) since $\mathbb{G}_n$ is profinite, 
$(K(E_n))^{h\mathbb{G}_n}$ denotes a continuous homotopy fixed point spectrum, and (c) $\pi_\ast(-)$ of the target of the above map is the abutment of a homotopy fixed point spectral sequence. For $n = 1$, $p \geq 5$, and $V_1 = V(1)$, we give a way to realize the above map and (c), by proving that $i_1$ induces a 
map \[K(L_{K(1)}(S^0)) \wedge v_{2}^{-1}V_1 \to (K(E_1) \wedge v_{2}^{-1}V_1)^{h\mathbb{G}_1},\] where the target of this map is a continuous homotopy fixed point spectrum, with an associated homotopy fixed point spectral sequence. Also, we prove that there is an equivalence
\[(K(E_1) \wedge v_{2}^{-1}V_1)^{h\mathbb{G}_1} 
\simeq (K(E_1))^{\widetilde{h}\mathbb{G}_1} \wedge v_2^{-1}V_1,\] 
where $(K(E_1))^{\widetilde{h}\mathbb{G}_1}$ is the homotopy fixed points with 
$\mathbb{G}_1$ regarded as a discrete group.
\end{abstract}

\maketitle
\section{Introduction}
\subsection{An overview of an Ausoni-Rognes conjecture and statements of our main theorems}
Let $n \geq 1$ and let $p$ be a prime. Let $E_n$ be the Lubin-Tate 
spectrum with 
$\pi_\ast(E_n) = W(\mathbb{F}_{p^n})\llbracket u_1, ..., u_{n-1} 
\rrbracket [u^{\pm 1}]$, where 
$W(\mathbb{F}_{p^n})$ is the ring of Witt vectors of the field 
$\mathbb{F}_{p^n}$ (with $p^n$ elements), the complete power series ring is in degree zero, 
and $|u| = 2$, and let 
$\mathbb{G}_n$ be the $n$th extended Morava stabilizer group. By 
\cite{Pgg/Hop0, AndreQuillen}, $E_n$ is a commutative $S$--algebra 
and the group $\mathbb{G}_n$ acts on $E_n$ by maps of 
commutative $S$--algebras. Given a commutative $S$--algebra $A$, 
the algebraic $K$--theory spectrum of $A$, $K(A)$, 
is a commutative $S$--algebra. Thus, $K(E_n)$ is a commutative 
$S$--algebra, and by 
the functoriality of $K(-)$, $\mathbb{G}_n$ acts on 
$K(E_n)$ by maps of commutative $S$--algebras.

Let $L_{K(n)}(S^0)$ denote the Bousfield localization of the sphere spectrum with 
respect to $K(n)$, the $n$th Morava $K$--theory spectrum. 
The group $\mathbb{G}_n$ is profinite, and by \cite{Rognes, joint}, the 
$K(n)$--local unit map 
\begin{equation}\zig\label{galois-extension}
L_{K(n)}(S^0) \rightarrow E_n\end{equation}
is a consistent 
profaithful $K(n)$--local 
profinite $\mathbb{G}_n$--Galois extension. 

Now let $V_n$ be a finite $p$--local complex of type $n+1$ and let 
$v \: \Sigma^d V_n \to V_n$ be a $v_{n+1}$--self-map, where $d$ is 
some positive integer (see \cite[Theorem 9]{nilpotencetwo}). 
The map $v$ induces a sequence 
\[V_n \to \Sigma^{-d}V_n \to \Sigma^{-2d} V_n \to \cdots\] of maps of 
spectra, and we let
\[v_{n+1}^{-1}V_n = \colim_{j \geq 0} \Sigma^{-jd}V_n,\] the colimit of the above sequence, 
denote the mapping telescope 
associated to the $v_{n+1}$--self-map $v$. As hinted at by the notation, the mapping 
telescope $v_{n+1}^{-1}V_n$ is independent of the choice of self-map $v$. 

In \cite[paragraph containing (0.1)]{acta}, \cite[Conjecture 4.2]{rognesguido}, and 
\cite[page 46; Remark 10.8]{jems}, Christian Ausoni and John Rognes 
conjectured that the $\mathbb{G}_n$--Galois extension 
$L_{K(n)}(S^0) \to E_n$ induces a map 
\begin{equation}\zig\label{conjecturalmap} 
K(L_{K(n)}(S^0)) \wedge v_{n+1}^{-1}V_n \rightarrow 
(K(E_n))^{h\mathbb{G}_n} \wedge v_{n+1}^{-1}V_n,\end{equation} 
where $(K(E_n))^{h\mathbb{G}_n}$ is a continuous homotopy 
fixed point spectrum, that is a weak equivalence, 
and associated with the target of this weak equivalence, there exists 
a homotopy fixed point spectral sequence that has the form
\[E_2^{s,t} = H^s_c\bigl(\mathbb{G}_n; (V_n)_t(K(E_n))[v_{n+1}^{-1}]\bigr) \Longrightarrow 
(V_n)_{t-s}((K(E_n))^{h\mathbb{G}_n})[v_{n+1}^{-1}],\] 
where the $E_2$--term is given by 
continuous cohomology. 
This conjecture is an extension of the Lichtenbaum-Quillen conjectures (for 
example, see \cite[(0.1), Theorem 4.1]{Thomason}), 
which can be viewed as corresponding to $n=0$ versions of the above (see 
\cite{rognesguido}, \cite[Section 10]{jems}). More generally, the conjecture is related to trying to understand \'etale descent for the algebraic $K$--theory of commutative $S$--algebras; for more details about this, see \cite[Introduction]{acta} and \cite[Section 4]{rognesicm}. 
\begin{Rk}
The above conjecture is just a piece of an important family of conjectures -- which include the chromatic redshift conjecture -- made by Ausoni and Rognes; we only state the part that we focus on in this paper. For more information about 
these conjectures, see \cite{jems, acta, rognesguido, publishedjems, rognesicm}. 
\end{Rk}

Notice that for every integer $t$, 
there is an isomorphism \[(V_n)_t(K(E_n))[v_{n+1}^{-1}] \cong 
\pi_t(K(E_n) \wedge v_{n+1}^{-1}V_n).\] Thus, when the above homotopy fixed point spectral sequence exists, since its abutment should be $\pi_\ast(-)$ 
of a homotopy fixed point spectrum, 
there should be an equivalence 
\begin{equation}\zig\label{conjecturalequivalence}
 (K(E_n))^{h\mathbb{G}_n} \wedge v_{n+1}^{-1}V_n \simeq 
(K(E_n) \wedge v_{n+1}^{-1}V_n)^{h\mathbb{G}_n},\end{equation} where 
the right-hand side is a continuous homotopy fixed point spectrum. 
Obtaining equivalence (\ref{conjecturalequivalence}) and a homotopy fixed point spectral 
sequence 
\[E_2^{s,t} = H^s_c(\mathbb{G}_n; \pi_t(K(E_n) \wedge v_{n+1}^{-1}V_n)) 
\Longrightarrow \pi_{t-s}\bigl((K(E_n) \wedge v_{n+1}^{-1}V_n)^{h\mathbb{G}_n}\bigr)\] immediately implies the existence of the spectral sequence in the 
above conjecture.
 
For making progress on this conjecture, one issue is that currently, 
for each $n$ and $p$, there are no published constructions of the continuous 
homotopy 
fixed point spectra 
\[(K(E_n))^{h\mathbb{G}_n}, \ (K(E_n) \wedge v_{n+1}^{-1}V_n)^{h\mathbb{G}_n}\] 
or the above two descent 
spectral sequences (here and below, we use the term ``descent spectral sequence" in place of 
``homotopy fixed point spectral sequence"). 

\begin{Rk}\label{condensed}
After doing the work in this paper, the author learned from Jacob Lurie that 
the condensed mathematics of Dustin Clausen and Peter Scholze can be 
used to define a notion of ``continuous homotopy fixed point spectrum" that 
is different from the one used for the results in this paper and it is not clear how these two notions are related. By viewing 
$E_n$ and $K(E_n)$ as condensed spectra, one can 
give a candidate definition of $(K(E_n))^{h\mathbb{G}_n}$ for all 
$n$ and $p$ in the setting of 
$\infty$--categories. In the case of $(K(E_1))^{h\mathbb{G}_1}$, it 
is unclear how this construction is related to the work in this paper. As 
suggested to the author by John Rognes, because 
of the close relationship between condensed objects and the pyknotic objects 
of Clark Barwick and Peter Haine (see 
the discussion in \cite[Section 0.3]{pyknotic}), it seems that 
by viewing $E_n$ and $K(E_n)$ as pyknotic spectra, there should be a pyknotic version of the above candidate definition of $(K(E_n))^{h\mathbb{G}_n}$ (see 
\cite[Section 3.1]{pyknotic}). The author is 
not aware of any other constructions of $(K(E_n))^{h\mathbb{G}_n}$, or of 
$(K(E_n) \wedge v_{n+1}^{-1}V_n)^{h\mathbb{G}_n}$ or the two 
spectral sequences. 
\end{Rk}

In this paper, in certain base cases, we address part of the issue described 
above: for $n=1$, $p \geq 5$, and $V_1 = V(1)$ -- the type $2$ Smith-Toda 
complex $S^0/(p, v_1)$,  
we construct the continuous homotopy fixed point spectrum
\[(K(E_1) \wedge v_{2}^{-1}V_{1})^{h\mathbb{G}_1}\] 
and we obtain the 
desired descent spectral sequence
\[E_2^{s,t} = H^s_c(\mathbb{G}_1; \pi_t(K(E_1) \wedge v_{2}^{-1}V_1)) \Longrightarrow 
\pi_{t-s}\bigl((K(E_1) \wedge v_{2}^{-1}V_1)^{h\mathbb{G}_1}\bigr).\] 
%
\begin{Rk}
Our work considers aspects of an Ausoni-Rognes conjecture involving the 
Galois extension $L_{K(n)}(S^0) \to E_n$, where the relevant group, $\mathbb{G}_n$, 
is infinite and profinite. 
For $K(n)$--local $G$--Galois extensions $A \to B$, where $G$ is a finite group, 
Ausoni and Rognes have made a conjecture similar to the one encapsulated above in 
(\ref{conjecturalmap}) 
\cite[Conjecture 4.2]{rognesguido}, and in these cases, since $G$ is naturally discrete, it is 
well-known that $(K(B))^{hG}$ always exists, and so there is no issue with the statement 
of the conjecture. For these cases, progress on the conjecture has been made by \cite{ClausenEtAl}.
\end{Rk}
Given our hypotheses -- $n = 1$, $p \geq 5$, and $V_1 = V(1)$, 
we can be a little more concrete about some of the 
main actors in the scenario that we focus on:
\[E_1 = KU_p,\] $p$--completed complex $K$--theory;  
\[\mathbb{G}_1 = \mathbb{Z}_p^\times,\] the group of units in the 
$p$--adic integers $\mathbb{Z}_p$; and \[v_2^{-1}V(1) = \colim_{j \geq 0}\Sigma^{-jd}V(1).\]
Then our first result is actually an extension of the 
aforementioned new $n =1$ constructions to all closed subgroups of $\mathbb{Z}_p^\times$. 

\begin{Thm}\label{appliedDSS}
Let $p \geq 5$. Given any closed subgroup $K$ of $\mathbb{Z}_p^\times$, there is a 
strongly convergent descent spectral sequence 
\[E_2^{s,t} = H^s_c(K ; \pi_t(K(KU_p) \wedge V(1))[v_2^{-1}]) \Longrightarrow 
\pi_{t-s}\bigl(\bigl(K(KU_p) \wedge v_2^{-1}V(1)\bigr)^{\mspace{-2mu} hK}\bigr),\]
with $E_2^{s,t} = 0$, for all $s \geq 2$ and any $t \in \mathbb{Z}$. Also, there 
is an equivalence of spectra 
\[\bigl(K(KU_p) \wedge v_2^{-1}V(1)\bigr)^{\mspace{-2mu} hK} 
\simeq \colim_{j \geq 0} \mspace{1mu}(K(KU_p) \wedge \Sigma^{-jd}V(1))^{hK}.\]
\end{Thm}

In the above result, the subgroup 
$K$ is a profinite group and each application of $(-)^{hK}$ denotes a 
continuous homotopy fixed point spectrum (as in \cite{joint}; we recall the 
definition later), 
formed in the setting of symmetric spectra of simplicial sets. 

Our next two results are about 
$\bigl(K(KU_p) \wedge v_2^{-1}V(1)\bigr)^{\mspace{-2mu}h\mathbb{Z}_p^\times}$.

\begin{Thm}\label{theoremaboutmap}
When $p \geq 5$, there is a canonical map
\[K(L_{K(1)}(S^0)) \wedge v_2^{-1}V(1) \to \bigl(K(KU_p) \wedge 
v_2^{-1}V(1)\bigr)^{\mspace{-2mu}h\mathbb{Z}_p^\times},\] induced by the 
$K(1)$--local unit map $L_{K(1)}(S^0) \to KU_p$,  
in the category of symmetric spectra.
\end{Thm}

For $n = 1$, $p \geq 5$, and $V_1 = V(1)$, if 
(\ref{conjecturalequivalence}) were valid, then Theorem \ref{theoremaboutmap} would yield the map 
in (\ref{conjecturalmap}), as a map in the stable homotopy category. 
Thus, in these cases, we hope that the spectral sequence of Theorem \ref{appliedDSS}, with 
$K = \mathbb{Z}_p^\times$, will be a useful computational tool for determining if the map in (\ref{conjecturalmap}) is a weak equivalence. 

Before stating the next result, we recall that if $G$ is any profinite group and $X$ is a (naive) $G$--spectrum, 
then $G$ can be regarded as a discrete group and one can always form the 
``discrete homotopy fixed point spectrum"
\[X^{\widetilde{h}G} = \mathrm{Map}_G(EG_+, X)\] 
(the usual notation 
for $X^{\widetilde{h}G}$ omits the ``$\,\,\widetilde{ \ }\,\,,$" but we use it here to distinguish $(-)^{\widetilde{h}G}$ from the continuous $(-)^{hG}$).

\begin{Thm}\label{surprising} 
When $p \geq 5$, there is an equivalence of spectra
\[\bigl(K(KU_p) \wedge v_2^{-1}V(1)\bigr)^{\mspace{-1.5mu} h\mathbb{Z}_p^\times} 
\simeq (K(KU_p))^{\widetilde{h}\mathbb{Z}_p^\times} \wedge v_2^{-1}V(1).\]
\end{Thm}

\begin{Rk} 
It is worth pointing out that in proving Theorem \ref{surprising}, we show that (for $p \geq 5$) there is a map 
\[\colim_{j \geq 0} (K(KU_p) \wedge \Sigma^{-jd}V(1))^{\widetilde{h} \mathbb{Z}_p^\times} \xrightarrow{\,\simeq\,} 
\colim_{j \geq 0} (K(KU_p) \wedge \Sigma^{-jd}V(1))^{h\mathbb{Z}_p^\times}\] that is a weak equivalence.
\end{Rk}

In (\ref{conjecturalequivalence}), when $n=1$, if $(K(E_1))^{h\mathbb{G}_1} = (K(KU_p))^{h\mathbb{Z}_p^\times}$ is changed to $(K(KU_p))^{\widetilde{h}\mathbb{Z}_p^\times}$, then 
Theorem \ref{surprising} is an instance of this ``modified (\ref{conjecturalequivalence})." But we do 
not take this observation as evidence that $(K(KU_p))^{\widetilde{h}\mathbb{Z}_p^\times}$ should be a definition of $(K(KU_p))^{h\mathbb{Z}_p^\times}$ for some $p$.

The proofs of Theorems \ref{appliedDSS}, \ref{theoremaboutmap}, and 
\ref{surprising} are given in the first part of Section \ref{section-proof}, that section's 
second part, 
and Section \ref{section-surprising}, respectively.

\subsection{The construction of the continuous homotopy fixed point spectra in Theorem \ref{appliedDSS}}
We now explain our work in more detail. 
Let $G$ be a profinite group and let $X$ be a $G$--spectrum. Then there is 
$X^{\widetilde{h}G}$ and one 
can always form the associated descent 
spectral sequence
\[E_2^{s,t} = H^s(G; \pi_t(X)) \Longrightarrow \pi_{t-s}(X^{\widetilde{h}G}),\] with 
$E_2$--term given by (non-continuous) group cohomology. However, it is 
not $(K(E_n) \wedge v_{n+1}^{-1}V_n)^{\widetilde{h}\mathbb{G}_n}$ that 
the conjecture of Ausoni and Rognes is concerned with. Since $\mathbb{G}_n$ is 
profinite and the $E_2$--term of the conjectured spectral sequence is 
given by continuous cohomology, one wants a continuous homotopy fixed point spectrum 
$(K(E_n) \wedge v_{n+1}^{-1}V_n)^{h\mathbb{G}_n}$ that takes the 
profinite topology of $\mathbb{G}_n$ into account; that is, we would like to know that 
$K(E_n) \wedge v_{n+1}^{-1}V_n$ is a continuous $\mathbb{G}_n$--spectrum 
in some sense, and that $(K(E_n) \wedge v_{n+1}^{-1}V_n)^{h\mathbb{G}_n}$ can 
be formed with respect to the continuous action.

To address this problem in the $n = 1, \ p \geq 5$ case, given a profinite 
group $G$, we work with discrete 
$G$--spectra (as in \cite{joint}) 
within the framework of symmetric spectra of simplicial sets (for more 
detail, see the end of the introduction). For the moment, let $X$ be a 
discrete $G$--spectrum. Then for all $k, l \geq 0$, the set of $l$-simplices 
of the $k$th pointed simplicial set of $X$, $X_{k,l}$, is a discrete 
$G$--set. Also, the homotopy fixed point spectrum $X^{hG}$ is defined 
(in \cite{joint}, as recalled at the end of the introduction) in a way that respects the profinite topology of 
$G$. Throughout this paper, we use $(-)^{hG}$ for these continuous homotopy 
fixed points, except for occurrences of ``$(K(E_n))^{h\mathbb{G}_n}$" and several other expressions where the context makes the meaning clear 
(all exceptions occur in the introduction). 

The following convention and terminology (from \cite{comenetz}) 
will be helpful to us.

\begin{Def}\label{fspectrum}
Let $X$ be a spectrum (that is, a symmetric spectrum). 
By ``$\,\pi_\ast(X)$," we always mean the homotopy groups 
\[\pi_t(X) := [S^t, X], \ \ \ t \in \mathbb{Z},\] of morphisms $S^t \to X$ in the homotopy category 
of symmetric spectra, where here, $S^t$ denotes a fixed cofibrant and fibrant model 
for the $t$-th suspension of the sphere spectrum. 
\end{Def}

\begin{Def}[{\cite[page 5]{comenetz}}]
A spectrum $X$ is an {\em $f$--spectrum} 
if $\pi_t(X)$ is finite for every integer $t$. 
\end{Def} 

Recall that a profinite group is strongly complete if every subgroup 
of finite index is open. Let $p$ be any prime: 
since $\mathbb{Z}_p$ is strongly complete, 
it follows that the profinite group $\mathbb{Z}_p \times H$, where $H$ is any finite 
discrete group and $\mathbb{Z}_p \times H$ is equipped with the product 
topology, is strongly complete. Thus (see Remark \ref{definition}), if 
$M$ is any $(\mathbb{Z}_p \times H)$--module that is finite, then $M$ is a 
discrete $(\mathbb{Z}_p \times H)$--module. Then, as an 
immediate consequence of Theorem \ref{illustrate} -- the proof of which uses 
\cite{symondsetal} in a key way -- and 
our central result, Theorem \ref{mainresult}, we have the following.

\begin{Thm}\label{intromain}
Let $p$ be any prime and let $H$ be any finite discrete group. 
If $X$ is a $(\mathbb{Z}_p \times H)$--spectrum and an $f$--spectrum, 
then $X$ is a discrete $(\mathbb{Z}_p \times H)$--spectrum.
\end{Thm} 

We state the conclusion of the above result more precisely: under the 
hypotheses of Theorem \ref{intromain}, there is a zigzag 
\begin{equation}\zig\label{intro-zigzag}
X \xrightarrow{\,\simeq\,} X' \xleftarrow{\,\simeq\,} X^\mathrm{dis}_\mathcal{N}
\end{equation}
of $(\mathbb{Z}_p \times H)$--spectra and 
$(\mathbb{Z}_p \times H)$--equivariant maps that are weak equivalences of 
symmetric spectra, and 
$X^\mathrm{dis}_\mathcal{N}$ is a discrete 
$(\mathbb{Z}_p \times H)$--spectrum. Thus, as in Definition \ref{def:homotopy}, it is natural to identify $X$ with 
the discrete $(\mathbb{Z}_p \times H)$--spectrum $X^\mathrm{dis}_\mathcal{N}$ 
and to define 
\[X^{h(\mathbb{Z}_p \times H)} = (X^\mathrm{dis}_\mathcal{N})^{h 
(\mathbb{Z}_p \times H)}.\] 

To go further, we need to introduce some notation and make a few 
comments. Let $\Sigma\mathrm{Sp}$ denote the model 
category of symmetric spectra (as in \cite[Theorem 3.4.4]{HSS}). We use 
\[(-)_f \: \Sigma\mathrm{Sp} \rightarrow 
\Sigma\mathrm{Sp}, \ \ \ Z \mapsto Z_f\] 
to denote a fibrant replacement functor, so that given the spectrum $Z$, 
there is a natural map $Z \rightarrow Z_f$ that is a trivial cofibration, 
with $Z_f$ 
fibrant. It is useful to note that if $X$ is a $G$--spectrum, then 
$X_f$ is also 
a $G$--spectrum and the trivial cofibration 
$X \rightarrow X_f$ is $G$--equivariant. Similarly, if $p \: X \to Y$ is a 
map of $G$--spectra (thus, $p$ is $G$--equivariant), then 
$p_f \: X_f \to Y_f$ is a map of $G$--spectra. 

We want to highlight the fact that in zigzag (\ref{intro-zigzag}), the 
construction of $X^\mathrm{dis}_\mathcal{N}$ is elementary: by Definition 
\ref{makeitdiscrete}, 
\begin{equation}\zig\label{simple}
X^\mathrm{dis}_\mathcal{N} = \colim_{m \geq 0} 
\holim_{[n] \in \Delta} 
\bigl(\mspace{1mu}\underbrace{\mathrm{Sets}(\mathbb{Z}_p \times H, \, \cdots , \mathrm{Sets}(
\mathbb{Z}_p \times H,}_{(n+1) \ \text{times}} 
X_f\mspace{-3mu}\underbrace{)\cdots\mspace{-.5mu})}_{\substack{\scriptscriptstyle{(n+1)}\\\scriptscriptstyle{\mathrm{times}}}}\mspace{.8mu}\bigr)^{\negthinspace 
(p^m\mathbb{Z}_p) \times \{e\}},\end{equation}
where each $(p^m\mathbb{Z}_p) \times \{e\}$ is an (open normal) subgroup of 
$\mathbb{Z}_p \times H$ and $p^m\mathbb{Z}_p$ has its usual meaning. We would 
like the reader to see how accessible the construction of $X^\mathrm{dis}_\mathcal{N}$ 
is, and thus, in this introduction, we do not think it is necessary to give any further explanation of 
(\ref{simple}). It turns out that for a $(\mathbb{Z}_p \times H)$--spectrum $X$ 
that is an 
$f$--spectrum, 
\begin{equation}\zig\label{simple-fixed}
X^{h(\mathbb{Z}_p \times H)} \simeq 
\Bigl(\holim_{[n] \in \Delta} 
\underbrace{\mathrm{Sets}(\mathbb{Z}_p \times H, \, \cdots , \mathrm{Sets}(
\mathbb{Z}_p \times H,}_{(n+1) \ \text{times}} X_f\mspace{-3mu}\underbrace{)\cdots\mspace{-.5mu})}_{\substack{\scriptscriptstyle{(n+1)}\\\scriptscriptstyle{\mathrm{times}}}}\mspace{.8mu}\Bigr)^{\negthinspace 
\mathbb{Z}_p \times H},\end{equation} by 
Theorem \ref{cool}. 
We are confident that without any additional explanation, the 
reader has at least an almost complete understanding of the meaning 
of the expression in (\ref{simple-fixed}); later reading about its precise 
definition (and 
that of (\ref{simple})) will mostly just confirm the 
reader's ``native conclusions."

We now explain our application of Theorem \ref{intromain} to the 
conjecture of Ausoni and Rognes. Let $p \geq 5$. Then 
\[\mathbb{Z}_p^\times \cong \mathbb{Z}_p \times 
\mathbb{Z}/(p-1),\] and as discussed earlier, $K(KU_p)$ is a $\mathbb{Z}_p^\times$--spectrum. By giving $V(1)$ the trivial $\mathbb{Z}_p^\times$--action, 
$K(KU_p) \wedge V(1)$ is a $\mathbb{Z}_p^\times$--spectrum under the diagonal action. 

Let $ku_p$ be the $p$--completed 
connective complex $K$--theory spectrum, with coefficients 
$\pi_\ast(ku_p) = \mathbb{Z}_p[u]$, 
where $|u| = 2$, as before. In \cite{localization}, Andrew Blumberg and Michael 
Mandell proved a conjecture of Rognes that there is a localization 
cofiber sequence 
\begin{equation}\zig\label{justKcofiber}
K(\mathbb{Z}_p) \rightarrow K(ku_p) \rightarrow K(KU_p) 
\rightarrow \Sigma K(\mathbb{Z}_p),\end{equation} 
and hence, there is a cofiber sequence 
\begin{equation}\zig\label{secondcofiber}
K(\mathbb{Z}_p) \wedge V(1) \rightarrow K(ku_p) \wedge V(1) 
\rightarrow K(KU_p) \wedge V(1) 
\rightarrow \Sigma (K(\mathbb{Z}_p) \wedge V(1)).\end{equation} By \cite{asterisque}, it is 
known that $K(\mathbb{Z}_p) \wedge V(1)$ is an $f$--spectrum (see 
also \cite[pages 663--664]{ausoniinventiones} for a helpful 
discussion about $V(1)_\ast K(\mathbb{Z}_p)$). 
Also, Ausoni \cite[Theorems 1.1, 8.1]{ausoniinventiones} 
showed that there exists an element $b \in V(1)_{2p+2}K(ku_p)$ 
such that if \[\mathbb{F}_p[b] \subset V(1)_\ast K(ku_p)\] denotes the polynomial $\mathbb{F}_p$--subalgebra generated by $b$, then there is a 
short exact sequence of graded $\mathbb{F}_p[b]$--modules 
\[0 \to \Sigma^{2p-3}\mathbb{F}_p \to V(1)_\ast K(ku_p) \to F \to 0,\] 
where $F$ is a free $\mathbb{F}_p[b]$--module on $4p+4$ generators. (Work of Rognes with Ausoni played a role in the Ausoni result: for example, see \cite[Section 8]{RognesNotes}. Also, \cite[Theorems 1.1, 8.1]{ausoniinventiones} were, in some sense, anticipated by \cite[discussion of Lemma 6.6]{BDR}, as explained in \cite[discussion of Proposition 1.4]{ausoniinventiones}.) 

It 
follows from the last result that 
$K(ku_p) \wedge V(1)$ is an $f$--spectrum, and hence, 
cofiber sequence (\ref{secondcofiber}) implies that $K(KU_p) \wedge V(1)$ is an 
$f$--spectrum. Therefore, by setting $H = \mathbb{Z}/{(p-1)}$ in 
Theorem \ref{intromain}, we obtain that 
$K(KU_p) \wedge V(1)$ is (in the sense of zigzag (\ref{intro-zigzag})) a discrete $\mathbb{Z}_p^\times$--spectrum. 

\begin{Rk}
Given our conclusion that $K(KU_p) \wedge V(1)$ is an 
$f$--spectrum for $p \geq 5$, it is natural to wonder if, for an arbitrary prime $p$, $K(E_n) \wedge V_n$ is an $f$--spectrum for $n \geq 2$. 
A starting point for considering this question would be a 
cofiber sequence analogous to the one in (\ref{justKcofiber}). 
For $n \geq 2$,  $E(n)_p$, the $p$--completion 
of the Johnson-Wilson spectrum $E(n)$, and $E_n$ are closely 
related, and in \cite[page 5]{acta}, Ausoni and Rognes state that they expect there to be such 
a cofiber sequence involving $K(E(n)_p)$ (for a precise description 
of this sequence, see [ibid.]). But by \cite{ABG}, such cofiber 
sequences do not exist. However, as Blumberg and Mandell discuss in 
\cite[Introduction]{BlumbergMandellMemoir}, there is a localization 
cofiber sequence 
\[K(\pi_0(E_n)) \to K(BP_n) \to K(E_n) \to \Sigma K(\pi_0(E_n)),\] where $BP_n$ is 
the connective cover of $E_n$, and we see that it has the attractive feature 
that $K(E_n)$ itself appears as a term, 
instead of $K(E(n)_p)$ (see 
\cite[Introduction]{BlumbergMandellMemoir} for more detail about 
this sequence). Thus, this cofiber sequence provides a way to 
begin studying the above question (the author has not pursued the argument 
suggested by cofiber sequences (\ref{justKcofiber}) and (\ref{secondcofiber})). 
\end{Rk}

We continue with letting $p \geq 5$. Our next step is to note that there is an equivalence
\begin{align*}
K(KU_p) \wedge v_2^{-1}V(1) & = K(KU_p) \wedge \bigl(\colim_{j \geq 0} 
\Sigma^{-jd}V(1)\bigr) \\ 
& \simeq \colim_{j \geq 0}\bigl(K(KU_p) \wedge \Sigma^{-jd}V(1)\bigr)_{\negthinspace f},
\end{align*}
where $\bigl\{\bigl(K(KU_p) \wedge \Sigma^{-jd}V(1)\bigr)_{\negthinspace f}\bigr\}_{\negthinspace 
j \geq 0}$ is a diagram of $\mathbb{Z}_p^\times$--spectra and $\mathbb{Z}_p^\times$--equivariant maps (as in the case of $V(1)$, each spectrum $\Sigma^{-jd}V(1)$ is given the trivial $\mathbb{Z}_p^\times$--action). Since $K(KU_p) \wedge V(1)$ is an $f$--spectrum, it is 
immediate that for each $j \geq 0$, $\bigl(K(KU_p) \wedge \Sigma^{-jd}V(1)\bigr)_{\negthinspace f}$ 
is an $f$--spectrum, and hence, Theorem \ref{intromain} 
implies that each $\bigl(K(KU_p) \wedge \Sigma^{-jd}V(1)\bigr)_{\negthinspace f}$ 
can be regarded as 
a discrete $\mathbb{Z}_p^\times$--spectrum. 

\begin{Rk}\label{terminology}
To aid the reader in making connections between the theory developed in this paper and the application of it that is discussed in this introduction, we use the terminology that is set up in later sections 
to express our main conclusions above 
(thus, $p \geq 5$). Let $\mathcal{N}$ denote the 
collection of open normal subgroups of $\mathbb{Z}_p^\times$ that corresponds to 
the family $\{(p^m\mathbb{Z}_p) \times \{e\}\}_{m \geq 0}$ of subgroups of 
$\mathbb{Z}_p \times \mathbb{Z}/(p-1)$. Then 
$\mathbb{Z}_p^\times$ has a good filtration (see Definition \ref{goodfiltration}), and 
we have shown that $(\mathbb{Z}_p^\times, K(KU_p) \wedge V(1), \mathcal{N})$ is a suitably finite triple (see Definition \ref{goodpair}) and 
\[\bigl(\mathbb{Z}_p^\times, \bigl\{\bigl(K(KU_p) \wedge \Sigma^{-jd}V(1)\bigr)_{\negthinspace f}\bigr\}_{\negthinspace j \geq 0}, \mathcal{N}\bigr)\] is a 
suitably filtered triple (Definition \ref{hyperdef}).
\end{Rk}

Let $\mathcal{N}$ be as defined in Remark \ref{terminology}. As explained 
(in greater generality) in the discussion centered around (\ref{zigzag}), there 
is a zigzag of 
$\mathbb{Z}_p^\times$--equivariant maps
\[\xymatrix@C=6pt@R=20pt{
{C_p : = \displaystyle{\colim_{j \geq 0}} \bigl(K(KU_p) \wedge \Sigma^{-jd}V(1)\bigr)_{\negthinspace f}}\, 
\ar@<0ex>[r]^-\simeq 
& \, {
\displaystyle
{\colim_{j \geq 0}}}\bigl(\bigl(K(KU_p) \wedge \Sigma^{-jd}V(1)\bigr)_{\negthinspace f}\bigr)' \\  
& {C_p^\mathrm{dis} := \displaystyle{\colim_{j \geq 0}}\bigl(\bigl(K(KU_p) \wedge \Sigma^{-jd}V(1)\bigr)_{\negthinspace f}\bigr)^{\negthinspace \mathrm{dis}}_{\negthinspace \mathcal{N}}} \ar@<0ex>[u]^-\simeq}
\] 
with each map a weak equivalence of symmetric spectra, and 
$C_p^\mathrm{dis}$ 
is a discrete $\mathbb{Z}_p^\times$--spectrum. The above zigzag is obtained by 
taking a colimit of the zigzags that are obtained from (\ref{intro-zigzag}) by setting 
$X$ (in (\ref{intro-zigzag})) equal to 
$\bigl(K(KU_p) \wedge \Sigma^{-jd}V(1)\bigr)_{\negthinspace f}$, for each 
$j \geq 0$. 

Let us now put 
the various equivalences above together. Following Definition 
\ref{naturalforcolimit}, we identify the $\mathbb{Z}_p^\times$--spectrum 
$C_p$ with the discrete 
$\mathbb{Z}_p^\times$--spectrum 
$C_p^\mathrm{dis}$ and we make the concomitant definition 
\[(C_p)^{h\mathbb{Z}_p^\times} = (C_p^\mathrm{dis})^{h\mathbb{Z}_p^\times}.\] 
Similarly, it is natural to identify 
the $\mathbb{Z}_p^\times$--spectrum $K(KU_p) \wedge v_2^{-1}V(1)$ with 
$C_p$, and hence, with 
the discrete $\mathbb{Z}_p^\times$--spectrum $C_p^\mathrm{dis}$ (the 
mapping telescope $v_2^{-1}V(1)$ has the trivial $\mathbb{Z}_p^\times$--action). 
Thus,
we define 
\[\bigl(K(KU_p) \wedge v_2^{-1}V(1)\bigr)^{\negthinspace h\mathbb{Z}_p^\times} 
= (C_p^\mathrm{dis})^{h\mathbb{Z}_p^\times}.\] More explicitly, we have 
\[\bigl(K(KU_p) \wedge v_2^{-1}V(1)\bigr)^{\negthinspace h\mathbb{Z}_p^\times} 
= \Bigl(\colim_{j \geq 0}\bigl(\bigl(K(KU_p) \wedge \Sigma^{-jd}V(1)\bigr)_{\negthinspace f}\bigr)^{\negthinspace \mathrm{dis}}_{\negthinspace \mathcal{N}}\Bigr)^{\negthinspace\mspace{-1mu} h\mathbb{Z}_p^\times}.\]

Now let $K$ be an arbitrary closed subgroup of $\mathbb{Z}_p^\times$. By 
the identification above of $K(KU_p) \wedge v_2^{-1}V(1)$ with 
$C_p^\mathrm{dis}$ in the world of 
$\mathbb{Z}_p^\times$--spectra and as in Definition \ref{naturalforcolimit}, 
it follows that the $K$--spectrum $K(KU_p) \wedge v_2^{-1}V(1)$ can 
be regarded as the discrete $K$--spectrum $C_p^\mathrm{dis}$, 
and hence, it is natural to define
\[\bigl(K(KU_p) \wedge v_2^{-1}V(1)\bigr)^{\negthinspace\mspace{.6mu} hK} 
= \Bigl(\colim_{j \geq 0}\bigl(\bigl(K(KU_p) \wedge \Sigma^{-jd}V(1)\bigr)_{\negthinspace f}\bigr)^{\negthinspace \mathrm{dis}}_{\negthinspace \mathcal{N}}\Bigr)^{\negthinspace\mspace{-1mu} hK}.\] Similarly (and easier; see the discussion just above (\ref{desiredequivtwo})), for 
each $j \geq 0$, it is natural to define
\[(K(KU_p) \wedge \Sigma^{-jd}V(1))^{hK} = \bigl((K(KU_p) \wedge \Sigma^{-jd}V(1))^\mathrm{dis}_\mathcal{N}\bigr)^{\negthinspace \mspace{.5mu}hK}.\]
This completes the 
construction of 
the continuous homotopy fixed point spectra that appear in 
Theorem \ref{appliedDSS}.  

\subsection{Concluding introductory remarks: our underlying framework, terminology, etc.}
In work in preparation, we use the theory developed in this paper to study 
$(KU_p)^{h\mathbb{Z}_p^\times}$, and more generally, $E_n^{hG}$, when 
$G$ is a closed subgroup of $\mathbb{G}_n$ that satisfies certain hypotheses 
(the spectra referred to here are continuous homotopy fixed point spectra, 
though $E_n$ is not a discrete $\mathbb{G}_n$--spectrum). 

We work in the framework of symmetric spectra in this paper because 
it is a symmetric monoidal category and such a category is important 
for studying the algebraic $K$--theory of commutative $S$--algebras. For 
example, in symmetric 
spectra, the role of commutative $S$--algebras is played by 
commutative symmetric ring spectra, and their properties 
are essential in the statement that $\mathbb{Z}_p^\times$ acts on $K(KU_p)$ 
by morphisms of commutative symmetric ring spectra. Furthermore, 
use of the framework of symmetric spectra makes available 
for future work the model category $\mathrm{Alg}_{A,G}$ of discrete 
commutative $G$--$A$--algebras, where $G$ is any profinite group and 
$A$ is a commutative symmetric ring spectrum (see \cite[Section 5.2]{joint}). Since 
the $\mathbb{G}_n$--action on $K(E_n)$ is by maps of commutative symmetric ring 
spectra, the model category $\mathrm{Alg}_{K(L_{K(n)}(S^0)),\mathbb{G}_n}$ (or 
$\mathrm{Alg}_{S^0,\mathbb{G}_n}$) might play a role in understanding 
$(K(E_n))^{h\mathbb{G}_n}$.   

\par 
We conclude this introduction with some preparatory comments for the upcoming work. For the rest of the paper, 
``spectrum" means symmetric spectrum of simplicial sets (except for a few instances in which the exception is clearly noted). It is useful to recall that given any collection 
$\{X_\gamma\}_{\gamma \in \Gamma}$ 
of fibrant spectra, there is an 
isomorphism $\pi_k\bigl(\prod_{\gamma \in \Gamma} 
X_\gamma\bigr) \cong 
\prod_{\gamma \in \Gamma} 
\pi_k(X_\gamma)$ of abelian groups, where $k$ is any 
integer, for the product of spectra 
$\prod_{\gamma \in \Gamma} X_\gamma$. Also, it is helpful to note 
that if a map $f$ of spectra is, when regarded as a map 
of Bousfield-Friedlander spectra, a weak equivalence (in the usual 
stable model structure on Bousfield-Friedlander spectra), then the 
map $f$ is a weak equivalence of spectra, by 
\cite[Theorem 3.1.11]{HSS}. We use $\holim$ to denote the homotopy 
limit for $\Sigma\mathrm{Sp}$, as defined in \cite[Definition 18.1.8]{hirschhorn}. 
\par
Let $G$ be any profinite group. 
A ``discrete $G$--spectrum" is a discrete 
symmetric $G$--spectrum, as defined 
in \cite[Section 2.3]{joint} (see also \cite[Section 3]{cts}); 
these objects, together with the $G$--equivariant 
maps (see \cite{joint} for the precise definition), 
constitute the category $\Sigma\mathrm{Sp}_G$ of discrete 
$G$--spectra. By \cite[Theorem 2.3.2]{joint}, 
there is a model category structure on $\Sigma\mathrm{Sp}_G$ 
in which a morphism $f$ in $\Sigma\mathrm{Sp}_G$ 
is a weak equivalence (cofibration) if and only if 
$f$ is a weak equivalence (cofibration) in $\Sigma\mathrm{Sp}$. Given 
a fibrant replacement functor \[(-)_{fG} \: \Sigma\mathrm{Sp}_G 
\rightarrow \Sigma\mathrm{Sp}_G, \ \ \ X \mapsto X_{fG}\] 
(thus, $X_{fG}$ is fibrant in $\Sigma\mathrm{Sp}_G$), such that  
there is a natural trivial cofibration $\eta \: X 
\rightarrow X_{fG}$ in $\Sigma\mathrm{Sp}_G$, there is the 
induced map \[\eta^G \: X^G \rightarrow (X_{fG})^G = X^{hG}.\] 
By \cite[Section 3.1]{joint}, the 
target of $\eta^G$, the homotopy fixed point spectrum 
$X^{hG},$ is the output of the right derived functor of fixed points.

\par
Given any profinite group $G$, a 
``$G$--spectrum" is a naive symmetric 
$G$--spectrum and not a genuine equivariant symmetric $G$--spectrum. Thus, 
when $G$ is finite, a $G$--spectrum need not be an equivariant symmetric 
$G$--spectrum in the sense of \cite{mandellequivariant} (defined by using the 
spheres $S(G) = \bigwedge_{G} S^1$ in the bonding maps). 
\vspace{.05in}
\par
\noindent
\textbf{Acknowledgements.} I thank John Rognes for helpful 
discussions. 
Also, I thank Christian Ausoni, Andrew Blumberg, Paul Goerss, Arturo 
Magidin, and Peter 
Symonds for useful comments, and Jacob Lurie for helpful conversations 
related to Remark \ref{condensed}. 
\section{Some preliminaries} 
\par
In this section, we explain some constructions and a result (Lemma 
\ref{e2term}) that will be useful for our main work later. 
As in the introduction, we let $G$ be any profinite group. 
\par
Given a set $S$, let $\mathrm{Sets}(G,S)$ be the $G$--set of all 
functions $f \: G \rightarrow S$, with $G$--action defined by 
\[(g \cdot f)(g') = f(g'g), \ \ \ g, g' \in G.\] Let $U$ be the forgetful functor from the 
category of $G$--sets to the category of sets. Then it is easy to see 
that $\mathrm{Sets}(G, -)$ is the right adjoint of $U$. By analogy with a 
standard construction in group cohomology, $\mathrm{Sets}(G,S)$ can 
be thought of as the ``coinduced $G$-set on $S$.'' 
\par
The 
construction $\mathrm{Sets}(G, S)$ prolongs to the category of 
$G$--spectra 
and the forgetful functor $U_G$ from the category of $G$--spectra to 
$\Sigma\mathrm{Sp}$ has a right adjoint that is given by 
the prolongation $\mathrm{Sets}(G, -)$, so that, given a spectrum $Z$ 
and any $k, l \geq 0$, the 
set of $l$-simplices of the pointed simplicial set $\mathrm{Sets}(G, Z)_k$ 
is defined by  
\[\mathrm{Sets}(G, Z)_{k,l} = \mathrm{Sets}(G, Z_{k,l}).\] Thus, 
for any $Z \in \Sigma\mathrm{Sp}$, there is an isomorphism 
\[\mathrm{Sets}(G, Z) \cong \textstyle{\prod_G Z}\] in $\Sigma\mathrm{Sp}$, 
where the right-hand side of the isomorphism is the product of $|G|$ copies of $Z$. Since the functors $U_G$ and 
$\mathrm{Sets}(G, - )$ are an adjoint pair, there is the associated triple 
(e.g., see \cite[8.6.2]{Weibel}), 
and, for any $G$-spectrum $X$, 
we let \[\mathrm{Sets}(G^{\bullet + 1}, X)\] denote the 
cosimplicial $G$--spectrum that is given in the usual way by the triple 
(for more detail, see \cite[8.6.4]{Weibel}).
\par
For any $m \geq 0$, we use 
$G^{m}$ to denote 
the Cartesian product of $m$ copies of $G$, with $G^0 = \ast$, the point. 
Then it is not hard to see that, for any $G$--spectrum $X$ and any 
$m \geq 0$, the ``$G$--spectrum of $m$--cosimplices" of the cosimplicial 
$G$--spectrum $\mathrm{Sets}(G^{\bullet+1}, X)$ satisfies the 
$G$--equivariant isomorphism
\[\mathrm{Sets}(G^{\bullet+1}, X)^m \cong \mathrm{Sets}(G, 
\mathrm{Sets}(G^{m}, X)),\] where, as before, $\mathrm{Sets}(G^{m}, X)$ 
is the spectrum defined on the level of sets by 
$\mathrm{Sets}(G^{m}, X)_{k,l} = \mathrm{Sets}(G^{m}, X_{k,l}),$ for 
every $k, l \geq 0$.
\par
We make no claim of originality for Lemma \ref{e2term} below; 
for example, it is a variation on the fact that if 
$L$ is a discrete group, $Z$ an $L$--spectrum that is fibrant 
in $\Sigma\mathrm{Sp}$, and $P$ 
a subgroup of $L$, then the descent spectral sequence 
\[E_2^{s,t} \Rightarrow 
\pi_{t-s}\bigl(\mathrm{Map}_P(EL_+, Z)\bigr) \cong 
\pi_{t-s}\bigl(Z^{\widetilde{h}P}\bigr)\] has an $E_2$--term that satisfies
\[E_2^{s,t} = H^s(P; \pi_t(Z)),\]
the (non-continuous) 
group cohomology of $P$ with coefficients in the $P$--module $\pi_t(Z)$. 
Also, the result below is a ``discrete version" of \cite[page 210 and the 
proof of Lemma 
5.4]{hGal} and \cite[proof of Lemma 7.12]{cts}. But, since Lemma \ref{e2term} 
is a useful tool for our work later, we give a complete proof.
\begin{Lem}\label{e2term}
Let $G$ be a profinite group. If $X$ is a $G$--spectrum and $K$ is a 
subgroup of $G$, then, for every $s \geq 0$ and any $t \in \mathbb{Z}$, there is an isomorphism
\[ \lims_\Delta \pi_t\bigl(\mathrm{Sets}(G^{\bullet+1}, X_f)^K\bigr) 
\cong H^s(K; \pi_t(X)).\] 
\end{Lem}
\begin{Rk}
To avoid any confusion, we note that in the statement of Lemma 
\ref{e2term}, $K$ is any subgroup of $G$ (thus, for example, 
$K$ does not have to be a closed subgroup of $G$).
\end{Rk}
\begin{proof}[{Proof of Lemma \ref{e2term}.}] 
If $A$ is an abelian group and $P$ is a profinite group, let 
$\mathrm{Sets}(P,A)$ be the abelian group of functions $P \rightarrow A$: 
in fact, $\mathrm{Sets}(P, A)$ is a $P$--module, with its $P$--action defined 
by $(p \cdot f)(p') = f(p'p)$. Then there is an isomorphism
\[ \lims_\Delta \pi_t\bigl(\mathrm{Sets}(G^{\bullet+1}, X_f)^K\bigr) \cong 
H^s\Bigl[\mathrm{Sets}(G^{\ast+1}, \pi_t(X))^K\Bigr],\]
where $\mathrm{Sets}(G^{\ast+1}, \pi_t(X))^K$ is the 
cochain complex obtained by applying, for each $m \geq 0$, the chain 
of isomorphisms
\begin{align*}
\pi_t\Bigl(\bigl(\mathrm{Sets}(G^{\bullet+1}, X_f)^K\bigr)^m\Bigr) & \cong 
\pi_t\bigl(\mathrm{Sets}(G, \mathrm{Sets}(G^{m}, X_f))^K\bigr)\\
& \cong \pi_t\Bigl(\textstyle{\prod}_{_{G/K}} \textstyle{\prod}_{_{G^m}} 
X_f\Bigr) \\ 
& \cong \textstyle{\prod}_{_{G/K}} \textstyle{\prod}_{_{G^m}} \pi_t(X_f) \\ 
& \cong \mathrm{Sets}(G, \mathrm{Sets}(G^m, \pi_t(X))^K \\
& \cong \mathrm{Sets}(G^{m+1}, \pi_t(X))^K.
\end{align*} Above, for $m \geq 1$, $\mathrm{Sets}(G^{m}, \pi_t(X))$ is the $K$--module 
of functions $G^m \rightarrow \pi_t(X)$ whose $K$--action is given by 
\[(k \cdot p)(g_1, g_2, g_3, ..., g_m) = p({g_1}k, g_2, g_3, ..., g_m),\]
for $k  \in  K,$ 
$p  \in  \mathrm{Sets}(G^{m}, \pi_t(X)),$ 
and $g_1, g_2, ..., g_m 
 \in  G.$ (In the preceding sentence, since $m \geq 1$, 
 it goes without saying that this sentence 
 also defines the $K$--action on the 
 $K$--module $\mathrm{Sets}(G^{m+1}, 
 \pi_t(X))$ that appears 
 in the last expression in the above chain of isomorphisms.)
\par
Notice that there is a $G$--equivariant monomorphism
\[\pi_t(X) \overset{\eta}{\longrightarrow} 
\mathrm{Sets}(G, \pi_t(X)), \ \ \ [f] \mapsto 
\bigl(g \mapsto g \cdot [f]\bigr)\] and a 
homomorphism \[\mathrm{ev}_{\negthinspace_1} \: \mathrm{Sets}(G, \pi_t(X)) 
\rightarrow \pi_t(X), \ \ \ p \mapsto p(1),\] such that 
$\mathrm{ev}_{\negthinspace_1} 
\negthinspace \circ \negthinspace 
\eta = \mathrm{id}_{\pi_t(X)}$. Then, since the 
cochain complex $\mathrm{Sets}(G^{\ast+1}, \pi_t(X))$ originally 
comes from a triple, there is 
an exact sequence 
\begin{equation}\label{longes}\zig
0 \rightarrow \pi_t(X) \overset{\eta}{\longrightarrow} 
\mathrm{Sets}(G^{\ast+1}, \pi_t(X))\end{equation} of $K$--modules
(for example, see the dual of 
\cite[Corollary 8.6.9]{Weibel}). 
\par
There is a chain 
\begin{align*}
\mathrm{Sets}(G, \mathrm{Sets}(G^m, \pi_t(X)) & 
\cong \textstyle{\prod_{_K} \prod_{_{G/K}} 
\negthinspace\mathrm{Sets}(G^m, \pi_t(X))} 
\\ &
\cong \mathrm{Hom}_{_{\scriptstyle{\mathbb{Z}}}}\negthinspace\bigl(\textstyle{\bigoplus_{\negthinspace_K} 
\negthinspace\mathbb{Z}}, 
\prod_{_{G/K}} \negthinspace\mathrm{Sets}(G^m, \pi_t(X))\bigr)\end{align*} 
of isomorphisms 
of $K$--modules, where  
$\mathrm{Hom}_{_{\scriptstyle{\mathbb{Z}}}}
\negthinspace\bigl(\textstyle{\bigoplus_{\negthinspace_K} \negthinspace\mathbb{Z}}, 
\prod_{_{G/K}} 
\negthinspace\mathrm{Sets}(G^m, \pi_t(X))\bigr)$ is a coinduced 
$K$--module, and hence, Shapiro's Lemma implies that
\begin{align*}
H^s\bigl(K; & 
\, \mathrm{Sets}(G, \mathrm{Sets}(G^m, \pi_t(X)))\bigr) 
\\ & \cong 
H^s\bigl(K; \mathrm{Hom}_{_{\scriptstyle{\mathbb{Z}}}}\negthinspace\bigl(\textstyle{\bigoplus_{\negthinspace_K} 
\negthinspace\mathbb{Z}}, 
\prod_{_{G/K}} \negthinspace\mathrm{Sets}(G^m, \pi_t(X))\bigr)\bigr) \\ 
& = 0, \end{align*}
whenever $s >0,$ 
for all $m \geq 0$. 
\par
Our last conclusion above implies that exact sequence (\ref{longes}) is a 
resolution of the $K$--module 
$\pi_t(X)$ by $(-)^K$--acyclic $K$--modules, and therefore, 
\[H^s\Bigl[\mathrm{Sets}(G^{\ast+1}, \pi_t(X))^K\Bigr] \cong 
H^s(K; \pi_t(X)),\] as desired.
\end{proof}
\section{Profinite groups that have a good filtration}
\par
As usual, let $G$ be a profinite group. In this section, after explaining 
the notion of a good filtration for $G$ and making several comments 
about it, we show that $\mathbb{Z}_p \times H$, where $p$ is any prime 
and $H$ is a finite discrete group, has a good filtration.

\begin{Def}\label{good}
Given a discrete 
$G$--module $M$, let 
\[\lambda^s_M \: H^s_c(G; M) \rightarrow H^s(G; M)\] be the natural 
homomorphism between continuous cohomology and 
non-continuous cohomology that is obtained by regarding 
each group $\mathrm{Map}_c(G^m, M)$ 
of continuous cochains as a subgroup of the corresponding 
group 
$\mathrm{Sets}(G^m, M)$ of all cochains. Then, in this paper (see 
Remark \ref{definition} below), 
we say that $G$ is {\it good} if $\lambda^s_M$ is 
an isomorphism for all $s \geq 0$ and 
every finite discrete $G$--module $M$.
\end{Def}
\begin{Rk}\label{definition}
The above definition is taken from 
\cite[\negthinspace 
page 13,~Exercise 2]{FrSerre}: if 
$G$ is strongly complete, so that $G \cong \widehat{G}$, where 
$\widehat{G}$ is the profinite completion of $G$, and 
$\lambda^s_M$ is an isomorphism for all $s \geq 0$ and every 
finite $G$--module $M$ (a finite $G$--module consists of finite orbits, so 
that every stabilizer subgroup of $G$ has finite index, and hence, is an 
open subgroup (since 
$G$ is strongly complete), so that a finite $G$--module is automatically 
a discrete 
$G$--module), then, following Serre, $G$ is ``bon." In 
general, since $G$ and $\widehat{G}$ need not be the same, our 
definition of ``good" is different from the usual one (that is, the 
aforementioned ``bon") in 
group theory. However, our use of ``good" in this paper should 
cause no confusion, since, throughout this paper, we only 
use ``good" in the sense of Definition \ref{good}.
\end{Rk} 
\par
We say that $G$ has 
{\em finite cohomological dimension} (``finite c.d.") if there exists some 
positive integer 
$r$ such that the continuous cohomology $H^s_c(G; M) = 0,$ for all discrete $G$--modules 
$M$, whenever $s > r$.
\begin{Def}\label{goodfiltration}
A profinite group $G$ {\it has a 
good filtration} if 
\begin{enumerate}
\item[(a)]
there exists a directed poset 
$\Lambda$ such that there is an inverse system 
\[\mathcal{N} = \{N_\alpha\}_{\alpha \in \Lambda}\] of 
open normal subgroups of $G$, with the maps in the diagram 
given by the inclusions (that is, $\alpha_1 \leq \alpha_2$ in $\Lambda$ 
if and only if $N_{{\alpha_2}}$ is a subgroup of $N_{{\alpha_1}}$);
\item[(b)]
the intersection $\bigcap_{\alpha \in \Lambda} \negthinspace N_\alpha$ 
is the trivial group $\{e\}$;
\item[(c)]
each $N_\alpha$ is a good profinite group, in the sense 
of Definition \ref{good}; and 
\item[(d)]
the collection $\{N_\alpha\}_{\alpha \in \Lambda}$ has 
uniformly bounded finite c.d.; that is, there exists a fixed natural number 
$r_{\negthinspace_G}$, such that 
$H^s_c(N_\alpha; M) = 0$, for all $s > r_{\negthinspace_G}$, whenever 
$\alpha \in \Lambda$ and $M$ is any discrete $N_\alpha$--module.
\end{enumerate}
\end{Def}
\begin{Rk}\label{yieldsgood} 
Let $G$ be a profinite group with a good filtration and let 
$\mathcal{N} = \{N_\alpha\}_{\alpha 
\in \Lambda}$ satisfy (a)--(d) in Definition \ref{goodfiltration}. 
It follows from (a) and (b) that $\mathcal{N}$ is 
a cofinal subcollection of the family of all open normal subgroups of $G$, and 
hence, the canonical homomorphism 
$G \rightarrow \lim_{\alpha \in \Lambda} G/N_\alpha$ is a 
homeomorphism. 
Now choose any $\alpha \in \Lambda$, so that 
$N_\alpha$ is good, by (c) above. We give an argument that is suggested by 
\cite[page 14, Exercise 2,~(c)]{FrSerre} (for instances of Serre's argument that are 
closely related to the version given here, see 
\cite[proof of Theorem 2.10]{symondsetal} and \cite[proof of Proposition 3.1]{Schroer}). 
Since \[\lambda^\ast_M \: H^\ast_c(N_\alpha;M) \to H^\ast(N_\alpha;M)\] is 
an isomorphism in each degree for any finite discrete $G$--module $M$, the 
$E_2$--term of the Lyndon-Hochschild-Serre spectral sequence 
\[E_2^{p,q} = 
H^p(G/N_\alpha; H^q_c(N_\alpha; M)) \Longrightarrow H^{p+q}_c(G; M)\] for 
continuous group cohomology (since $G/N_\alpha$ is a finite discrete group, the $E_2$--term 
is given by just group cohomology) is 
isomorphic to the $E_2$--term of the corresponding 
Lyndon-Hochschild-Serre spectral 
sequence 
\[H^p(G/N_\alpha; H^q(N_\alpha; M)) \Longrightarrow H^{p+q}(G; M)\] 
for group cohomology, and hence, by comparison of spectral sequences, the 
map \[\lambda^s_M \: H^s_c(G; M) \xrightarrow{\cong} H^s(G; M)\] is an isomorphism, for all $s \geq 0$ and 
any finite discrete $G$--module $M$. 
\end{Rk}
\begin{Rk}
Let $G$ be a profinite group that has finite c.d. and let $\{N_\alpha\}_{\alpha \in \Lambda}$ be an inverse system of open normal subgroups of $G$ that 
satisfies (a)--(c) in Definition \ref{goodfiltration}. Then the inverse system 
also satisfies (d), so that $G$ has a good filtration. This conclusion follows 
from the fact that for $r$ as in 
our definition of finite c.d. above (just before Definition \ref{goodfiltration}), 
Shapiro's Lemma implies that whenever $s > r$, given any $\alpha \in 
\Lambda$, 
\[H^s_c(N_\alpha; M) \cong H^s_c(G; 
\mathrm{Coind}^G_{N_\alpha}\negthinspace(M)) 
= 0,\] for all discrete 
$N_\alpha$--modules $M$ (above, $\mathrm{Coind}^G_{N_\alpha}\negthinspace(M)$ is the coinduced module of continuous functions $G \to M$ that are 
$N_\alpha$--equivariant). 
\end{Rk}

\begin{Thm}\label{illustrate}
Let $p$ be any prime and let $G = \mathbb{Z}_p \times H$, where $H$ is a finite discrete group and $G$ is equipped with the product topology. Then 
$G$ has a good filtration.
\end{Thm}
\begin{proof}
Recall that there is a descending chain 
\[\mathbb{Z}_p = U_0 \gneq U_1 \gneq \cdots \gneq U_m \gneq 
\cdots\] 
of open normal subgroups of $\mathbb{Z}_p$, with 
$U_m = p^m\mathbb{Z}_p$ for each $m \geq 0$ 
and $\bigcap_{m \geq 0} U_m = \{e\}$. For each $m \geq 0$, we set 
$N_m = U_m \times \{e\}$, a subgroup of $G$. 
We will show that $\{N_m\}_{m \geq 0}$ 
satisfies conditions (a)--(d) in Definition \ref{goodfiltration}.

It is easy to see that $\{N_m\}_{m \geq 0}$ satisfies (a) and (b). 
By \cite[Theorem 2.9]{symondsetal}, $\mathbb{Z}_p$ is a good 
profinite group and, for each $m \geq 0$, $N_m \cong \mathbb{Z}_p$, showing 
that (c) is valid. 
Finally, since the pro-$p$-group $\mathbb{Z}_p$ has cohomological 
$p$-dimension equal to one, it follows that $\mathbb{Z}_p$ has finite 
c.d. This fact, coupled with another application of the isomorphisms 
$N_m \cong \mathbb{Z}_p$ for all $m \geq 0$, shows that 
(d) holds. 
\end{proof}

\section{An $r$--$\mathbb{Z}_p$--spectrum is a 
discrete $\mathbb{Z}_p$--spectrum}\label{mainsection}
\par 
In this section, 
we prove one of the key results of this paper, Theorem \ref{mainresult}; 
the title above illustrates a special case of this result, and the unfamiliar 
term in the title is defined below.

\begin{Def}\label{torsion}
Let $G$ be a profinite group and $X$ a $G$--spectrum. 
If $\pi_t(X)$ is a 
finite {\it discrete} $G$--module for every $t \in \mathbb{Z}$, then we say that 
$X$ is an {\it $r$--$G$--spectrum} (in this term, the ``r" is for 
``restricted," which is, roughly speaking, a synonym of ``finite").
\end{Def} 

\begin{Rk}\label{rgfgrelation}
Since an $r$--$G$--spectrum is both a $G$--spectrum and an $f$--spectrum, 
our first thought was to use the term ``$f$--$G$--spectrum" for such an object, 
but this term is already used (often) by \cite{davisquick} (see [ibid., Definition 3.1]). 
If $G$ is strongly complete, then 
every $r$--$G$--spectrum $X$ has an $f$--$G$--spectrum associated to it in the following way: $X_f$ is a $G$--spectrum and since it is a fibrant spectrum, 
for each integer $t$, there is an isomorphism 
\begin{equation}\label{finiteabelian}\zig
\pi_t(X_f) \cong \colim_{k} \pi_{t+k}(X_k) = \pi_t(U(X_f))\end{equation} 
of finite 
abelian groups, where the 
last expression in (\ref{finiteabelian}) 
refers to the $t$-th (classical) stable homotopy group of the 
Bousfield-Friedlander spectrum $U(X_f)$ that underlies $X_f$, and 
hence, by an application of \cite[Theorem 5.15]{quickspectra}, there is a $G$--equivariant map and weak equivalence 
$U(X_f) \xrightarrow{\,\simeq\,} F^s_G(U(X_f))$ of 
Bousfield-Friedlander spectra, with $F^s_G(U(X_f))$ 
an $f$--$G$--spectrum.
\end{Rk}

For the remainder of this section (with the exception of Lemma 
\ref{lemma}), $G$ denotes a profinite group that has a good 
filtration. 
Thus, we let \[\mathcal{N} = \{N_\alpha\}_{\alpha \in \Lambda}\] 
be an inverse system of open normal subgroups of $G$ that 
satisfies the requirements of Definition \ref{goodfiltration}.
\begin{Def}\label{makeitdiscrete} Let $X$ be a $G$--spectrum. We set 
\[X^\mathrm{dis}_\mathcal{N} = \colim_{\alpha \in \Lambda} \holim_\Delta 
\mathrm{Sets}(G^{\bullet + 1}, X_f)^{N_\alpha},\] where the colimit 
is formed in $\Sigma\mathrm{Sp}$. %
\end{Def} 

Since each $N_\alpha$ is an open normal subgroup of $G$, with 
$G/N_\alpha$ 
a finite discrete group, $\mathrm{Sets}(G^{\bullet + 1}, X_f)^{N_\alpha}$ 
is a 
cosimplicial $G/N_\alpha$--spectrum. Thus, 
the spectrum 
$\textstyle{\holim_\Delta \mathrm{Sets}(G^{\bullet + 1}, X_f)^{N_\alpha}}$ is a 
$G/N_\alpha$--spectrum, and hence, a discrete 
$G$--spectrum (via the canonical projection $G \rightarrow G/N_\alpha)$. 
By \cite[Section 3.4]{joint}, 
colimits in $\Sigma\mathrm{Sp}_G$ are formed in $\Sigma\mathrm{Sp}$, 
and hence, we have the following observation.
\begin{Lem}\label{thisguy}
If $X$ is a $G$--spectrum, where $G$ is a profinite group that has a 
good filtration, then $X^\mathrm{dis}_\mathcal{N}$ is a discrete $G$--spectrum.
\end{Lem}

\begin{Rk}\label{neededlater}
Let $X$ be a $G$--spectrum. 
Since $\mathcal{N}$ is cofinal in the collection of all open normal subgroups 
of $G$, there is an isomorphism 
\[X^\mathrm{dis}_\mathcal{N} \cong \colim_{N \vartriangleleft_o G} 
\holim_\Delta \mathrm{Sets}(G^{\bullet + 1}, X_f)^N\] 
of discrete $G$--spectra, where above, 
$N \vartriangleleft_o G$ means that $N$ is an open normal subgroup 
of $G$. Similarly, if 
$\mathcal{N}' = \{N_{\alpha'}\}_{\alpha' \in \Lambda'}$ is another inverse 
system of open normal subgroups of $G$ that satisfies 
Definition \ref{goodfiltration}, there 
is an isomorphism 
\[X^\mathrm{dis}_{\mathcal{N}'} = 
 \colim_{\alpha' \in \Lambda'} 
\holim_\Delta \mathrm{Sets}(G^{\bullet + 1}, X_f)^{N_{\alpha'}} 
\cong \colim_{N \vartriangleleft_o G} 
\holim_\Delta \mathrm{Sets}(G^{\bullet + 1}, X_f)^N\] in 
$\Sigma\mathrm{Sp}_G$, and hence, 
there is an isomorphism 
$X^\mathrm{dis}_\mathcal{N} \cong X^\mathrm{dis}_{\mathcal{N}'}$ 
in $\Sigma\mathrm{Sp}_G$. It follows that the 
definition of $X^\mathrm{dis}_\mathcal{N}$ 
is independent of the choice of inverse system 
$\mathcal{N}$ up to isomorphism.
\end{Rk}

\par
Now we are ready to prove the central result of this paper: its 
conclusion can be abbreviated by saying that if $X$ is an 
$r$--$G$--spectrum (as in Definition \ref{torsion}), 
then $X$ is a discrete $G$--spectrum. 
We break up our work for this result into two pieces. 
The first piece, Lemma \ref{lemma} below, can be regarded as a 
special case of \cite[Proposition 6.4]{neisen}, in the setting 
of $G$--spectra.
\begin{Lem}\label{lemma}
If $G$ is any profinite group and $X$ is any $G$--spectrum, then 
there is a $G$--equivariant map
\[i_X \: X \overset{\simeq}{\longrightarrow} 
\holim_\Delta \mathrm{Sets}(G^{\bullet + 1}, X_f)\] that is a 
weak equivalence in $\Sigma\mathrm{Sp}$.
\end{Lem}
\begin{proof}
Given a spectrum $Z$, let $\mathrm{cc}^\bullet(Z)$ denote 
the constant cosimplicial spectrum on $Z$. 
Then the $G$--equivariant map $i_X$
is defined to be 
the composition 
\[i_X \: 
X \overset{\simeq}{\longrightarrow} 
X_f \overset{\cong}{\longrightarrow} \lim_\Delta \mathrm{cc}^\bullet(X_f) 
\rightarrow \holim_\Delta \mathrm{cc}^\bullet(X_f) 
\rightarrow \holim_\Delta \mathrm{Sets}(G^{\bullet+1}, X_f),\] where the 
last (rightmost) map is induced by repeated use of 
the $G$--equivariant monomorphism
$i \: Y \rightarrow \mathrm{Sets}(G, Y)$ of $G$--spectra, 
that is defined on the level of sets, for any $G$--spectrum $Y$, by the maps
\[Y_{k,l} \rightarrow \mathrm{Sets}(G, Y_{k,l}), \ \ \ y \mapsto (g \mapsto 
g \cdot y).\]
\par
Notice that for each $m \geq 0$, the spectrum of $m$-cosimplices of 
$\mathrm{Sets}(G^{\bullet+1}, X_f)$,
\[\bigl(\mathrm{Sets}(G^{\bullet+1}, X_f)\bigr)^m 
\cong \textstyle{\prod_{_{G^{m+1}}} \negthinspace X_f},\] is fibrant, so that 
$\mathrm{Sets}(G^{\bullet+1}, X_f)$ is a cosimplicial fibrant spectrum. 
Thus, there is a homotopy spectral sequence 
\begin{equation}\label{SSone}\zig
^{I\negthinspace}E_2^{s,t} \cong H^s\bigl[\pi_t(\mathrm{Sets}(G^{{\ast}+1}, X_f))\bigr] \Longrightarrow 
\pi_{t-s}\bigl(\holim_\Delta \mathrm{Sets}(G^{\bullet + 1}, X_f)\bigr).
\end{equation} By Lemma \ref{e2term}, we have
\[^{I\negthinspace}E_2^{s,t} \cong H^s(\{e\}; \pi_t(X)) =
\begin{cases} \pi_t(X), & s=0;\\ 0, & s>0,\end{cases}\] and hence, 
spectral sequence $^{I\negthinspace}E_r^{\ast, \ast}$ of (\ref{SSone}) 
collapses, showing that $i_X$ is a weak equivalence.
\end{proof}
\begin{Thm}\label{mainresult}
Let $G$ be a profinite group that has a good filtration and let 
$\mathcal{N}$ be a diagram of subgroups of $G$ that satisfies 
Definition \ref{goodfiltration}. If $X$ is an $r$--$G$--spectrum, 
then there is a zigzag of $G$--equivariant maps
\begin{equation}\label{zigzag2}\zig
X \overset{\simeq}{\longrightarrow} 
\holim_\Delta \mathrm{Sets}(G^{\bullet + 1}, X_f) 
\overset{\simeq}{\longleftarrow} 
X^\mathrm{dis}_\mathcal{N}\end{equation} 
that are weak equivalences in $\Sigma\mathrm{Sp}$. 
\end{Thm}
\begin{Rk}
As stated just before Lemma \ref{lemma}, 
the above theorem says that (given a 
suitable profinite group $G$) an $r$--$G$--spectrum 
can be regarded as a discrete $G$--spectrum (in a canonical way): 
the ``$G$--equivariant zigzag" of weak equivalences in (\ref{zigzag2}) 
makes this 
statement precise.
\end{Rk} 
\begin{proof}[{Proof of Theorem \ref{mainresult}}] 
\par
By Lemma \ref{lemma}, it suffices to construct a $G$--equivariant 
map \[\phi_{\negthinspace_X} \: X^\mathrm{dis}_\mathcal{N} = \colim_{\alpha \in \Lambda} 
\holim_\Delta \mathrm{Sets}(G^{\bullet + 1}, X_f)^{N_\alpha} 
\rightarrow \holim_\Delta \mathrm{Sets}(G^{\bullet + 1}, X_f)\] and 
then show that it is a weak equivalence of spectra. 
The $G$--equivariant map $\phi_{\negthinspace_X}$
is defined to be the composition
\[\varinjlim \holim_\Delta 
\mathrm{Ens}(G^{\bullet \scriptscriptstyle{+ 1}}\negthinspace\negthinspace, X_f)^{N_\alpha} 
\overset{\overset{\phi_{\negthinspace_{\scriptscriptstyle{X}}}^{\scriptscriptstyle 1}}{}}{\longrightarrow} 
\holim_\Delta \varinjlim \mathrm{Ens}(G^{\bullet \scriptscriptstyle{+ 1}}\negthinspace\negthinspace, 
X_f)^{N_\alpha}
\overset{\overset{\phi_{\negthinspace_{\scriptscriptstyle{X}}}^{\scriptscriptstyle 2}}{}}{\longrightarrow} 
\holim_\Delta \mathrm{Ens}(G^{\bullet \scriptscriptstyle{+ 1}}\negthinspace\negthinspace, X_f)\] 
of canonical maps, 
where, here (and below), to conserve space, 
we (sometimes) use the notation ``$\,\displaystyle{\varinjlim} \, $" to denote 
``$\, \displaystyle{\colim_{\alpha \in \Lambda}}\,$", 
and ``$\mathrm{Ens}$" in place 
of ``$\mathrm{Sets}$." 
\par
The definition of the map $\phi_{\negthinspace_X}^2$ is given 
explicitly as follows: 
the collection of 
inclusions $\mathrm{Sets}(G^{\bullet + 1}, X_f)^{N_\alpha} 
\hookrightarrow \mathrm{Sets}(G^{\bullet + 1}, X_f)$ induces the 
morphism 
\[\overline{\phi_{\negthinspace_X}^2} \: \colim_{\alpha \in \Lambda} \mathrm{Sets}(G^{\bullet + 1}, X_f)^{N_\alpha}
\rightarrow \mathrm{Sets}(G^{\bullet + 1}, X_f)\] 
of cosimplicial $G$--spectra, and 
$\phi_{\negthinspace_X}^2 = 
\smash{\displaystyle{\holim_\Delta 
\overline{\phi_{\negthinspace_X}^2}}}\,$. The morphism 
$\overline{\phi_{\negthinspace_X}^2}$ also induces a 
map \[E^{_{_{\scriptstyle{\ast,\ast}}}}_{r}\negthinspace\Bigl(\overline{\phi_{\negthinspace_X}^2}\Bigr) \:  
^{II\negthinspace}E_r^{\ast,\ast} \rightarrow 
{^{I\negthinspace}E_r^{\ast,\ast}},\]
from the homotopy spectral sequence
\begin{equation}\label{SStwo}\zig
^{II\negthinspace}E_2^{s,t} \negthinspace=\negthinspace 
H^s\Bigl[\pi_t\bigl(\varinjlim 
\mathrm{Ens}(G^{\ast {\scriptscriptstyle{+ 1}}}\negthinspace, 
X_f)^{N_\alpha}\bigr)\negthinspace\Bigr] 
\Rightarrow
\pi_{t-s}\Bigl(\holim_\Delta \varinjlim 
\mathrm{Ens}(G^{\bullet {\scriptscriptstyle{+ 1}}}\negthinspace, 
X_f)^{N_\alpha}\negthinspace\Bigr)\end{equation}
to spectral sequence (\ref{SSone}). We point out that the construction of 
spectral sequence 
(\ref{SStwo}) uses the fact that for each $m \geq 0$, the spectrum of 
$m$--cosimplices of $\varinjlim 
\mathrm{Ens}(G^{\bullet {\scriptscriptstyle{+ 1}}}\negthinspace, 
X_f)^{N_\alpha}$ satisfies
\[\bigl(\colim_{\alpha \in \Lambda} 
\mathrm{Sets}(G^{\bullet+1}, 
X_f)^{N_\alpha}\bigr)^m \cong \colim_{\alpha \in \Lambda} 
\bigl(\textstyle{\prod}_{_{G/N_\alpha}} \negthinspace \negthinspace
\prod_{_{G^m}} \negthinspace X_f\bigr),\] which is a 
fibrant spectrum, since products and 
filtered colimits of fibrant spectra are again fibrant 
(the second fact is justified, for example, in \cite[Section 5]{dualityprospectra}), 
so that $\varinjlim 
\mathrm{Ens}(G^{\bullet {\scriptscriptstyle{+ 1}}}\negthinspace, 
X_f)^{N_\alpha}$ is a cosimplicial fibrant spectrum.
\par
Notice that for spectral sequence $^{II\negthinspace}E_r^{\ast,\ast}$, 
there is the 
chain of isomorphisms
\begin{equation}\label{chain}\zig
\begin{split}
^{II\negthinspace}E_2^{s,t} & \cong \colim_{\alpha \in \Lambda} 
H^s(N_\alpha; \pi_t(X)) \\
& \cong \colim_{\alpha \in \Lambda} H^s_c(N_\alpha; \pi_t(X)) \\
& \cong H^s_c\bigl(\,\textstyle{\bigcap}_{\alpha \in \Lambda} 
N_\alpha; \pi_t(X)\bigr) 
\\ & = H^s(\{e\}; \pi_t(X)),\end{split}\end{equation} 
where the first isomorphism uses 
 Lemma \ref{e2term} and 
 the fact that filtered colimits of fibrant spectra commute 
with $[S^t, -]$; the second isomorphism 
applies the assumption 
that each $N_\alpha$ is a good profinite group; and 
the last step (involving the equality) is due to property (b) of Definition 
\ref{goodfiltration}.
Therefore, there is an isomorphism 
\[^{II\negthinspace}E_2^{s,t} {\, \cong}{} \, {^{I\negthinspace}
E_2^{s,t}},\] for all $s$ and $t$, so that the 
map $E^{_{_{\scriptstyle{\ast,\ast}}}}_{r}\negthinspace\Bigl(\overline{\phi_{\negthinspace_X}^2}\Bigr)$ of spectral sequences is an isomorphism 
from the $E_2$--terms onward. Hence, the map 
$\pi_\ast(\phi^2_{\negthinspace_X}) = 
[S^\ast, \phi^2_{\negthinspace_X}]$ between the abutments 
of these conditionally convergent spectral sequences is an 
isomorphism, so that $\phi_{\negthinspace_X}^2$ is a weak equivalence.
\par
As in (\ref{chain}), there are isomorphisms
\begin{equation}\label{2ndspot}\zig
H^s\Bigl[\pi_t\bigl(\mathrm{Sets}(G^{\ast+1},X_f)^{N_\alpha}\bigr)\Bigr] 
\cong 
H^s(N_\alpha; \pi_t(X)) \cong H^s_c(N_\alpha; \pi_t(X))\end{equation} 
for each $\alpha$, and hence, 
condition (d) of Definition \ref{goodfiltration} implies 
that \begin{equation}\label{usingfinitecd}\zig
H^s\Bigl[\pi_t\bigl(\mathrm{Sets}(G^{\ast+1},X_f)^{N_\alpha}\bigr)\Bigr] 
= 0, \ \ \mathrm{for \ 
all} \ s> r_G, \ \mathrm{every} \ t \in \mathbb{Z}, \ \mathrm{and \ each \ 
\alpha}.\end{equation}
Therefore, the map 
$\phi_{\negthinspace_X}^1$ is a weak equivalence, by \cite[Proposition 3.4]{Mitchell}. 
\par
Finally, 
we can conclude that $\phi_{\negthinspace_X}$ is a weak equivalence, since 
$\phi_{\negthinspace_X}^1$ and $\phi_{\negthinspace_X}^2$ are weak equivalences.
\end{proof}
\section{An extension of the central result, Theorem \ref{mainresult}}
\par
In this section, we show -- in Theorem \ref{looser} -- that the hypotheses of Theorem 
\ref{mainresult} can be slightly loosened. We give this result in this later 
section so that Theorem \ref{mainresult} (and Section \ref{mainsection}) 
is ready-made for the intended applications. Suppose that 
$X$ is a $G$--spectrum 
with homotopy groups that are always torsion discrete $G$--modules: 
as explained in the second part of this section, the homotopy groups of 
such a $G$--spectrum 
are closely related to those of 
$r$--$G$--spectra. However, we 
explain why our proof of Theorem \ref{mainresult} does not extend 
to this more general ``torsion case."  
\par
For the rest of this section, we suppose that $G$ is 
an arbitrary profinite group and $X$ is any $G$--spectrum. 
Given this context, 
it is easy to see that Definition \ref{makeitdiscrete} and Lemma 
\ref{thisguy} depend only on condition (a) of Definition 
\ref{goodfiltration}, and hence, under only the additional 
assumption of condition (a), the spectrum $X^\mathrm{dis}_\mathcal{N}$ is 
defined and is a discrete $G$--spectrum. Also, 
the proof of Theorem \ref{mainresult} depends only 
on 
\begin{itemize}
\item[(i)]
condition (a);  
\item[(ii)] the assumption that the $G$--module 
$\pi_t(X)$ is a discrete $G$--module, for every $t \in \mathbb{Z}$; and
\item[(iii)]
part (b) of Definition \ref{goodfiltration}: $\bigcap_{\alpha \in \Lambda} 
\negthinspace N_\alpha = \{e\}$, 
\end{itemize}
{\it except} in three spots:
\begin{itemize}
\item 
in the second 
isomorphisms of (\ref{chain}) and (\ref{2ndspot}), 
in addition to (i) and (ii) above, the 
proof of Theorem \ref{mainresult} uses both 
the assumption that $\pi_t(X)$ is finite for every integer $t$ 
and part (c) of Definition \ref{goodfiltration}; and 
\item 
in (\ref{usingfinitecd}), besides (i) and (ii) above, the proof uses 
part (d) of Definition \ref{goodfiltration}. 
\end{itemize}
These observations imply the 
following result.
\begin{Thm}\label{looser}
Let $G$ be a profinite group, with $\mathcal{N} = \{N_\alpha\}_{\alpha \in \Lambda}$ an inverse system of open normal subgroups of $G$ 
that satisfies properties 
$\negthinspace\mathrm{(\mathit{a}) \, \mathit{and} \, (\mathit{b})}$ 
of Definition \ref{goodfiltration}, and let $X$ be a $G$--spectrum 
such that condition $\mathrm{(\mathit{ii})}\negthinspace$ above holds. 
Also, suppose that the map 
\[\lambda^s_{\pi_t(X)} \: H^s_c(N_\alpha; \pi_t(X)) \rightarrow H^s(N_\alpha; \pi_t(X))\] 
is an isomorphism for all $s \geq 0$, every integer $t$, and each 
$\alpha \in \Lambda$. If 
\begin{itemize}
\item
there exists a natural number $r$, such 
that for all integers $t$ and every $\alpha \in \Lambda$, 
$H^s_c(N_\alpha; \pi_t(X)) = 0$, for all $s > r$; or 
\item
there exists 
some fixed integer $l$, such that $\pi_t(X) = 0$, for all $t > l$, 
\end{itemize}
then there is a zigzag of $G$--equivariant maps
\[X \overset{\simeq}{\longrightarrow} 
\holim_\Delta \mathrm{Sets}(G^{\bullet + 1}, X_f) 
\overset{\simeq}{\longleftarrow} 
X^\mathrm{dis}_\mathcal{N}\]
that are weak equivalences in $\Sigma\mathrm{Sp}$, with
$X_\mathcal{N}^\mathrm{dis}$ 
$\mathrm{(}$defined as in Definition 
\ref{makeitdiscrete}$\mathrm{)}$ a discrete $G$--spectrum.
\end{Thm}
\begin{proof}
The only part of the theorem that is not justified by the remarks 
preceding it is the following. In our proof of Theorem \ref{mainresult}, 
in (\ref{usingfinitecd}), we assumed condition (d) of Definition \ref{goodfiltration}, but by 
\cite[Proposition 3.4]{Mitchell}, an alternative to assuming condition 
(d) is to require that there exists some fixed integer $l$, such that 
for each $m \geq 0$ and every $\alpha \in \Lambda,$ 
\[\pi_t\bigl(\mathrm{Sets}(G^{m+1}, X_f)^{N_\alpha}\bigr) 
\cong \textstyle{\prod_{_{G/{N_\alpha} \times G^m}} \negthinspace \pi_t(X)} 
= 0, \ \ \mathrm{for \ all} \ t>l,\] which is equivalent 
to assuming that $\pi_t(X) = 0$, for all $t > l.$ 
\end{proof}
\par 
We conclude this section by explaining why the proof of 
Theorem \ref{mainresult} fails to extend to the case when 
$X$ is a $G$--spectrum with each homotopy group a (possibly 
infinite) discrete $G$--module that is also a torsion abelian group. 
With $G$ as in Theorem \ref{mainresult}, our 
assumptions imply that for each $t \in \mathbb{Z}$, 
\[\pi_t(X) = \textstyle{\bigcup_{_\beta} \negthinspace 
M_{t,\beta}}\] is the union 
of its finite $G$--submodules $M_{t,\beta}$, each of which is 
automatically a discrete $G$--module. 
\par
As discussed at the beginning of this section, 
in the second isomorphisms in 
(\ref{chain}) and (\ref{2ndspot}), we need to know that 
for each $\alpha$ and every integer $t$,
the natural map 
\[\lambda^s_{\pi_t(X)} \: 
H^s_c(N_\alpha; \pi_t(X)) \rightarrow H^s(N_\alpha; \pi_t(X))\] 
is an isomorphism, for all $s \geq 0$. Since each $N_\alpha$ 
is a good profinite group, 
there are isomorphisms
\begin{align*}
H^s_c(N_\alpha; \pi_t(X)) & \cong \colim_\beta 
H^s_c(N_\alpha; M_{t, \beta})\\
& \cong \colim_\beta H^s(N_\alpha; M_{t,\beta}) \\
& \cong H^s\Bigl[\colim_\beta 
\mathrm{Sets}(N_\alpha^{\ast}, M_{t,\beta})\Bigr],
\end{align*} 
where, here, given an $N_\alpha$--module $M$, 
$\mathrm{Sets}(N_\alpha^\ast, M)$ denotes 
the usual cochain complex such that 
$H^s\bigl[\mathrm{Sets}(N_\alpha^\ast, M)\bigr] = 
H^s(N_\alpha; M)$, for each $s \geq 0$, with 
the abelian group of $m$--cochains equal to 
\[\mathrm{Sets}(N_\alpha^\ast, M)^m = \mathrm{Sets}(N_\alpha^m, 
M) \cong 
\textstyle{\prod_{_{N^m_\alpha}} \negthinspace M}, \ \ \mathrm{for \ 
each} \ m \geq 0.\] 
It follows that the map 
$\lambda^s_{\pi_t(X)}$ is an isomorphism if and only if 
the canonical map
\[h^{s,t} \: H^s\Bigl[\colim_\beta 
\mathrm{Sets}(N_\alpha^{\ast}, M_{t,\beta})\Bigr] \rightarrow 
H^s\Bigl[\mathrm{Sets}(N_\alpha^{\ast}, 
\textstyle{\bigcup_{_\beta} \negthinspace M_{t,\beta})}\Bigr] = 
H^s(N_\alpha; \pi_t(X))\] is an isomorphism.
\par
Since filtered colimits and infinite products do not commute in general, 
the map $h^{s,t}$ above need not be an isomorphism, so that 
$\lambda^s_{\pi_t(X)}$ need not be an isomorphism: this situation 
is the crux of 
what prevents the proof of Theorem \ref{mainresult} from going 
through in the case when each $\pi_t(X)$ is a torsion discrete $G$--module.
\begin{Rk}
Let $G$ be as in Theorem \ref{mainresult} and suppose that 
$X$ is a $G$--spectrum such that $\pi_t(X)$ is a discrete $G$--module 
and torsion abelian group, for every integer $t$. Then it 
is clear from the above discussion that if $G$, as an abstract 
group, is of type $FP_\infty$ (for background on this notion, we refer to \cite{KennethBrown}), then $H^\ast(G; -) \cong 
\mathrm{Ext}^\ast_{\mathbb{Z}[G]}(\mathbb{Z}, -)$ commutes with direct 
limits, and hence, the conclusion of Theorem \ref{mainresult} is still 
valid. 
%
Now we add the desirable condition that 
$G$ is an infinite group, and we give an argument 
that we learned from Peter Symonds. As an abstract group, if 
$G$ is of type $FP_\infty$, then it is of type $FP_1$, and hence, it is 
finitely generated (abstractly) and thereby countably infinite, 
contradicting the fact that $G$ must be uncountable (since it is 
profinite). Therefore, $G$ cannot be both infinite and, abstractly, of type 
$FP_\infty$. 
%
\end{Rk}
\section{The spectrum $X^\mathrm{dis}_\mathcal{N}$, fibrancy, and 
homotopy fixed points}\label{fibrancy}
\par
In this section, we let $G$ be any profinite group and $X$ any 
$G$--spectrum.
\begin{Def}\label{goodpair}
If $G$, $X$, and $\mathcal{N}$ (an inverse system of open normal subgroups of $G$) satisfy the hypotheses of Theorem \ref{mainresult} or 
Theorem \ref{looser}, then we say that the 
triple $(G,X, \mathcal{N})$ is {\it suitably finite}. (In the preceding sentence, by 
satisfying the hypotheses of Theorem \ref{looser}, we mean that 
$G$, $X$, and $\mathcal{N}$ 
satisfy the conditions of the first two sentences of 
Theorem \ref{looser} and at least one of the two ``itemized conditions" (that is, the conditions marked by a ``$\,\bullet \,$") listed in the third sentence of 
Theorem \ref{looser}.) Notice that if $(G,X, \mathcal{N})$ is a suitably finite triple, then there is a 
zigzag of $G$--equivariant maps 
\[X \overset{\simeq}{\longrightarrow} 
\holim_\Delta \mathrm{Sets}(G^{\bullet + 1}, X_f) 
\overset{\simeq}{\longleftarrow} 
X^\mathrm{dis}_\mathcal{N}\]
that are weak equivalences in $\Sigma\mathrm{Sp}$.
\end{Def}
\begin{Def}\label{def:homotopy}
If $(G,X, \mathcal{N})$ is a suitably finite triple, then because of the above zigzag
of equivalences between $X$ and $X^\mathrm{dis}_\mathcal{N}$, it is 
natural to identify $X$ with the discrete $G$--spectrum 
$X^\mathrm{dis}_\mathcal{N}$, and hence, to define
\[X^{hG} = (X^\mathrm{dis}_\mathcal{N})^{hG}.\] 
\end{Def}

\begin{Rk}
Let $(G, X, \mathcal{N})$ be a suitably finite triple, with the inverse system 
$\mathcal{N}$ written as $\{N_\alpha\}_{\alpha \in \Lambda}$, and 
suppose that $X$ is a discrete $G$--spectrum 
(that is, before the identification of Definition \ref{def:homotopy}, 
$X \in \Sigma\mathrm{Sp}_G$). In this case, after following 
Definition \ref{def:homotopy}, $X^{hG}$ can mean $(X_{fG})^G$ or 
$(X^\mathrm{dis}_\mathcal{N})^{hG}$. Since 
$X \in \Sigma\mathrm{Sp}_G$, the weak equivalence $i_X \: 
X \xrightarrow{\,\simeq\,} \holim_\Delta \mathrm{Sets}(G^{\bullet+1}, X_f)$ 
factors into the map $\delta \: X \to X^\mathrm{dis}_\mathcal{N}$, which is 
defined to be the composition 
\[X \xrightarrow{\,\cong\,} \colim_{\alpha \in \Lambda} X^{N_\alpha} 
\xrightarrow{\,\underset{\alpha \in \Lambda}{\colim} (i_X)^{N_\alpha}\,} 
\colim_{\alpha \in \Lambda} \bigl(\holim_\Delta \mathrm{Sets}(G^{\bullet+1}, 
X_f)\bigr)^{\negthinspace \mspace{1mu} N_\alpha} \xrightarrow{\,\cong\,} X^\mathrm{dis}_\mathcal{N}\] (the first isomorphism in the composition is 
due to the fact that, since 
$\mathcal{N}$ satisfies (a) and (b) in Definition \ref{goodfiltration}, 
$\mathcal{N}$ is a cofinal subcollection of $\{N\,|\,N \vartriangleleft_o G\}$), 
followed by the weak equivalence 
$X^\mathrm{dis}_\mathcal{N} \xrightarrow{\,\simeq\,} 
\holim_\Delta \mathrm{Sets}(G^{\bullet+1}, X_f)$, and hence, the map $\delta$ is a weak 
equivalence of spectra. It follows that $\delta$ is a weak equivalence in 
$\Sigma\mathrm{Sp}_G$; therefore, $\delta$ induces a weak equivalence 
$(X_{fG})^G \xrightarrow{\,\simeq\,} (X^\mathrm{dis}_\mathcal{N})^{hG}$, 
showing that the two possible interpretations of $X^{hG}$ are equivalent 
to each other.   
\end{Rk}

Several interesting consequences of Definition \ref{def:homotopy} are 
stated in Theorem \ref{cool} below. Before giving this result, we 
need to give some background material for its proof. 

Let $\gspt$ be the category of $G$--spectra (as defined at the end of the 
introduction): $\gspt$ has a model category structure in which a morphism 
$f$ is a weak equivalence (cofibration) if and only if $f$ is a weak equivalence 
(cofibration) when regarded as a morphism in $\Sigma\mathrm{Sp}$. The 
existence of this model structure 
follows, for example, from the fact that $\gspt$ is isomorphic to $\Sigma\mathrm{Sp}^{\{\ast_G\}}$, the 
category of functors $\{\ast_G\} \to \Sigma\mathrm{Sp}$, where 
$\{\ast_G\}$ is the one-object groupoid associated to $G$, and this 
diagram category can be equipped with an injective model structure, 
by \cite[Proposition A.2.8.2]{luriebook}, since $\Sigma\mathrm{Sp}$ is a 
combinatorial model category. 

Since the forgetful functor $U_G \: \gspt \to \Sigma\mathrm{Sp}$ 
preserves weak equivalences and cofibrations, the adjoint functors 
$(U_G, \mathrm{Sets}(G,-))$ are a Quillen pair. Also, it will be helpful 
to recall the standard fact that if $Y$ is fibrant in $\gspt$, then $Y$ is 
fibrant in $\Sigma\mathrm{Sp}$ (since, for example, an injective fibrant 
object in $\Sigma\mathrm{Sp}^{\{\ast_G\}}$ is projective fibrant in 
$\Sigma\mathrm{Sp}^{\{\ast_G\}}$ (one reference for this is \cite[Remark A.2.8.5]{luriebook}; $\Sigma\mathrm{Sp}^{\{\ast_G\}}$ has a projective model structure 
by [ibid., Proposition A.2.8.2])).  

The left adjoint functor $\Sigma\mathrm{Sp} \to \gspt$ 
that sends a spectrum to itself, but now regarded as a $G$--spectrum with trivial $G$--action, preserves weak equivalences and 
cofibrations. It follows that the right adjoint, the $G$--fixed points 
functor $(-)^G \: \gspt \to \Sigma\mathrm{Sp}$, is a 
right Quillen functor, and if $Y \to Y_\mathrm{fib}$ is a trivial cofibration 
to a fibrant object, in $\gspt$, then 
\[Y^{\widetilde{h}G} = (Y_\mathrm{fib})^G.\]  

As in \cite[Example 1.1.5.8]{luriebook}, the category $\{\ast_G\}$ can 
be regarded as a simplicial category by defining the simplicial set 
$\mathrm{Map}_{\{\ast_G\}}(\ast_G, \ast_G)$ to be the constant simplicial set on $\mathrm{Hom}_{\{\ast_G\}}(\ast_G, \ast_G)$. With 
$\mathcal{S}$ equal to the category of simplicial sets, 
it is easy to see that the category of $\mathcal{S}$-enriched functors 
from $\{\ast_G\}$ to the simplicial category 
$\Sigma\mathrm{Sp}$, with morphisms the 
$\mathcal{S}$-enriched natural transformations, is identical to the usual functor category $\Sigma\mathrm{Sp}^{\{\ast_G\}}$. Since 
$\Sigma\mathrm{Sp}$ is a simplicial model category, it follows from 
\cite[Proposition A.3.3.2, Remark A.3.3.4]{luriebook} that the 
injective model structure on $\Sigma\mathrm{Sp}^{\{\ast_G\}}$ is 
simplicial, and hence, the model category $\gspt$ is simplicial. 

Let $\holimG$ denote the homotopy limit for $\gspt$, as defined in 
\cite[Definition 18.1.8]{hirschhorn} (this definition uses the 
fact that $\gspt$ is a simplicial model category). Since the forgetful functor 
$U_G$ is a right adjoint (its left adjoint is given by the functor 
$\Sigma\mathrm{Sp} \to \gspt$ that sends a spectrum $Z$ to 
the $G$--spectrum $\bigvee_G Z$, where $G$ acts only on the 
indexing set of the coproduct), limits in $\gspt$ are formed in 
$\Sigma\mathrm{Sp}$. Also, it is a standard fact that the cotensor 
$Y^{S_{\typicaldot}}$ in $\gspt$, where $Y$ is a $G$--spectrum and 
$S_{\text{\tiny \textbullet}}$ is a 
simplicial set, is equal to the corresponding cotensor $Y^{S_{\typicaldot}}$ 
in $\Sigma\mathrm{Sp}$ equipped with the natural $G$--action. Since $\holimG$ 
is defined as the equalizer of maps between products of cotensors, 
it follows that $\holimG$ is formed in $\Sigma\mathrm{Sp}$: if 
$\{Y_c\}_{c \in \mathcal{C}}$ is a small diagram of $G$--spectra, then 
 $\holim^G_{\mathcal{C}} \{Y_c\}_c$ is equal to the spectrum 
 $\holim_{\mathcal{C}} \{Y_c\}_c$ equipped with the induced 
 $G$--action.

Now we recall \cite[Theorem 4.3]{2ndnyjm}, but we 
rewrite it for symmetric spectra ([loc. cit.] is written 
in the world of Bousfield-Friedlander spectra, but the argument is 
the same when using symmetric spectra). The forgetful functor $U \: \Sigma\mathrm{Sp}_G \to \gspt$ 
has a right adjoint, the discretization functor 
\[(-)_d \: \gspt \to \Sigma\mathrm{Sp}_G, \ \ \ Y \mapsto (Y)_d 
= \colim_{N \vartriangleleft_o G} Y^N.\] Since $U$ preserves 
weak equivalences and cofibrations, the functors $(U, (-)_d)$ are a Quillen pair.

\begin{Thm}\label{cool}
If $(G, X, \mathcal{N})$ is a suitably finite triple, then 
\[X^{hG} \simeq (X^\mathrm{dis}_\mathcal{N})^G 
\cong \bigl(\holim_\Delta 
\mathrm{Sets}(G^{\bullet + 1}, X_f)\bigr)^{\negthinspace G} \simeq X^{\widetilde{h}G}.\]
\end{Thm}

\begin{proof}
Since $X_f$ is fibrant in $\Sigma\mathrm{Sp}$, $\mathrm{Sets}(G, X_f)$ is 
fibrant in $\gspt$, and hence, it is fibrant 
in $\Sigma\mathrm{Sp}$. By iterating this 
argument, we obtain that $\mathrm{Sets}(G^{\bullet + 1}, X_f)$ 
is a cosimplicial fibrant $G$--spectrum (that is, 
for each $m \geq 0$, the $m$--cosimplices are a fibrant $G$--spectrum). 
It follows that $\holim^G_\Delta \mathrm{Sets}(G^{\bullet + 1}, X_f)$ is a 
fibrant $G$--spectrum. Since $\holim^G_\Delta \mathrm{Sets}(G^{\bullet + 1}, X_f)$ is equal to the $G$--spectrum $\holim_\Delta \mathrm{Sets}(G^{\bullet + 1}, X_f)$, we write the latter instead of the former. 

Let $X \to X_\mathrm{fib}$ be a trivial cofibration 
to a fibrant object, in $\gspt$, and notice that the equivalence 
$X \xrightarrow{\simeq} 
\holim_\Delta \mathrm{Sets}(G^{\bullet + 1}, X_f)$ (in Definition \ref{goodpair}) is a weak 
equivalence with fibrant target, in $\gspt$. Then there exists a weak equivalence 
\[X_\mathrm{fib} \xrightarrow{\simeq} 
\holim_\Delta \mathrm{Sets}(G^{\bullet + 1}, X_f)\] in $\gspt$, and since 
$(-)^G \: \gspt \to \Sigma\mathrm{Sp}$ is a right Quillen functor, 
the induced map
\[X^{\widetilde{h}G} = (X_\mathrm{fib})^G \xrightarrow{\simeq} 
\bigl(\holim_\Delta \mathrm{Sets}(G^{\bullet + 1}, X_f)\bigr)^{\negthinspace G}\] is a weak equivalence. 

Since $\mathcal{N}$ 
satisfies conditions (a) and (b) 
in Definition \ref{goodfiltration}, there is an isomorphism
\[ 
X^\mathrm{dis}_\mathcal{N} 
\cong \colim_{N \vartriangleleft_o G} 
\bigl(\holim_\Delta \mathrm{Sets}(G^{\bullet + 1}, X_f)\bigr)^{\negthinspace N} 
= \bigl(\holim_\Delta \mathrm{Sets}(G^{\bullet + 1}, X_f)\bigr)_{\negthinspace d}
\] of discrete $G$--spectra, as noted in Remark \ref{neededlater}, 
and since the functor $(-)_d$ is a 
right Quillen functor, $\bigl(\holim_\Delta \mathrm{Sets}(G^{\bullet + 1}, X_f)\bigr)_{\negthinspace d}$ is a fibrant discrete $G$--spectrum, and 
hence, so is $X^\mathrm{dis}_\mathcal{N}$. Thus, applying the 
right Quillen functor $(-)^G \: \Sigma\mathrm{Sp}_G \to \Sigma\mathrm{Sp}$ to the fibrant replacement map 
$X^\mathrm{dis}_\mathcal{N} \to (X^\mathrm{dis}_\mathcal{N})_{fG}$, 
which is a trivial cofibration between fibrant objects in 
$\Sigma\mathrm{Sp}_G$, 
yields the weak equivalence
\[(X^\mathrm{dis}_\mathcal{N})^G \xrightarrow{\simeq} 
\bigl((X^\mathrm{dis}_\mathcal{N})_{fG}\bigr)^{\negthinspace G} 
= (X^\mathrm{dis}_\mathcal{N})^{hG} = X^{hG}.\]

The final step is to note that 
\[(X^\mathrm{dis}_\mathcal{N})^G \cong 
\Bigl(\mspace{1mu}\colim_{N \vartriangleleft_o G} 
\bigl(\holim_\Delta \mathrm{Sets}(G^{\bullet + 1}, X_f)\bigr)^{\negthinspace N}\Bigr)^{\mspace{-6mu} G} 
\cong \bigl(\holim_\Delta \mathrm{Sets}(G^{\bullet + 1}, X_f)\bigr)^{\negthinspace G},\] as desired. 
\end{proof}

\begin{Rk}
Let $(G, X, \mathcal{N})$ be a suitably finite triple. In light of the 
proof of Theorem \ref{cool}, we reexamine the $G$--equivariant 
zigzag \[X \xrightarrow{\,\simeq\,} 
\holim_\Delta \mathrm{Sets}(G^{\bullet+1}, X_f) \xleftarrow{\,{\simeq}\,} X^\mathrm{dis}_\mathcal{N}\] of equivalences: 
the first map is taking an explicit 
fibrant replacement of $X$ -- call it $X'$ -- in the model category of 
$G$--spectra (here, we do not require the fibrant replacement map to be a 
cofibration) and the second map is the inclusion into $X'$ 
from its largest discrete $G$--subspectrum $X^\mathrm{dis}_\mathcal{N} 
\cong (X')_d$ 
(this description of the output of the functor $(-)_d$ is used in \cite[page 861]{toriimodule} and it 
is meant to be taken literally: if $X''$ is a discrete $G$--subspectrum of $X'$, then 
the isomorphism \[X'' \cong \colim_{N \vartriangleleft_o G} (X'')^N\] shows that 
$X''$ is a $G$--subspectrum of 
\[\colim_{N \vartriangleleft_o G} (X')^N = (X')_d \cong X^\mathrm{dis}_\mathcal{N}).\] Therefore, we can think of the above zigzag as saying that $X$ is 
equivalent to an explicit model -- $X^\mathrm{dis}_\mathcal{N}$ -- for $(\mathbf{R}(-)_d)(X)$, the output of the total right derived functor $\mathbf{R}(-)_d$ of $(-)_d$ 
(recall from the proof of Theorem \ref{cool} that 
there is a weak equivalence $X_\mathrm{fib} \xrightarrow{\,\simeq\,} X'$ between 
fibrant objects in $\gspt$, and hence, there is a weak equivalence 
\[(\mathbf{R}(-)_d)(X) = (X_\mathrm{fib})_d \xrightarrow{\,{\simeq}\,} (X')_d \xrightarrow{\,\cong\,} X^\mathrm{dis}_\mathcal{N}\] of discrete 
$G$--spectra).
\end{Rk}

We can use Theorem \ref{cool} to build a descent spectral 
sequence, as follows.

\begin{Cor}\label{descentSS}
If $(G, X, \mathcal{N})$ is a suitably finite triple, then there 
is a conditionally convergent descent spectral sequence that has 
the form
\[E_2^{s,t} \cong H^s_c(G; \pi_t(X)) \cong H^s(G; \pi_t(X)) 
\Longrightarrow \pi_{t-s}(X^{\widetilde{h}G}) \cong 
\pi_{t-s}(X^{hG}).\]
\end{Cor} 

\begin{proof}
At the beginning of the proof of Theorem \ref{cool}, we noted 
that $\mathrm{Sets}(G^{\bullet+1}, X_f)$ is a cosimplicial 
fibrant $G$--spectrum, and hence, 
$\mathrm{Sets}(G^{\bullet+1}, X_f)^G$ is a cosimplicial fibrant 
spectrum. Thus, there is a homotopy spectral sequence 
\[E_2^{s,t} = H^s\bigl[\pi_t\bigl(\mathrm{Sets}(G^{\ast+1}, 
X_f)^G\bigr)\bigr] 
\Longrightarrow \pi_{t-s}\bigl(\mspace{1mu}\holim_\Delta 
\mathrm{Sets}(G^{\bullet+1}, X_f)^G\bigr).
\] This spectral sequence is the descent spectral sequence 
described in the corollary, and the isomorphism that occurs 
in the abutment of the descent spectral sequence follows 
immediately from applying Theorem \ref{cool} to the abutment 
of the above homotopy spectral sequence. 

Lemma \ref{e2term} yields the isomorphism $E_2^{s,t} 
\cong H^s(G; \pi_t(X))$, for all $s \geq 0$ and any integer $t$. 
If the triple $(G, X,\mathcal{N})$ satisfies the hypotheses of 
Theorem \ref{mainresult}, then there is an isomorphism 
$H^s_c(G; \pi_t(X)) \cong H^s(G; \pi_t(X))$, for all $s \geq 0$ and 
any $t \in \mathbb{Z}$, by Remark 
\ref{yieldsgood}. If the triple $(G, X, \mathcal{N})$ 
satisfies the hypotheses of 
Theorem \ref{looser}, then this same isomorphism is obtained 
by applying the spectral sequence argument of Remark 
\ref{yieldsgood} to the case where the ``$M$" in the remark 
is changed to $\pi_t(X)$.
\end{proof}

To illustrate the previous result, we have the following 
special case for $G = \mathbb{Z}_p$.

\begin{Cor}
Let $p$ be any prime. If $X$ is a $\mathbb{Z}_p$--spectrum 
and an $f$--spectrum, 
then 
there is a strongly convergent descent spectral sequence 
\[E_2^{s,t} = H^s_c(\mathbb{Z}_p; \pi_t(X)) 
\Longrightarrow \pi_{t-s}(X^{h\mathbb{Z}_p}),\] with 
$E_2^{s,t} = 0$, whenever $s \geq 2$ and $t$ is any integer.
\end{Cor}

\begin{proof}
By Theorem \ref{illustrate}, $\mathbb{Z}_p$ has a good filtration, 
with $\mathcal{N} = \{p^m\mathbb{Z}_p\}_{m \geq 0}$. Any 
subgroup of finite index in $\mathbb{Z}_p$ is open in 
$\mathbb{Z}_p$ and $\pi_\ast(X)$ is finite in each degree, so 
that $\pi_t(X)$ is a discrete $\mathbb{Z}_p$--module 
(see Remark \ref{definition}), for every 
integer $t$. It follows that $X$ is an $r$--$G$--spectrum, and hence, 
$(\mathbb{Z}_p, X, \{p^m\mathbb{Z}_p\}_{m \geq 0})$ is a suitably 
finite triple, $X$ can be identified with the discrete $\mathbb{Z}_p$--spectrum $X^\mathrm{dis}_\mathcal{N}$, $X^{h\mathbb{Z}_p}$ is 
defined, and Corollary \ref{descentSS} gives the conditionally convergent 
spectral sequence described above. 

Since each $\pi_t(X)$ is finite and $\mathbb{Z}_p$ has 
cohomological $p$-dimension one, 
$E_2^{s,t} = H^s_c(\mathbb{Z}_p; \pi_t(X)) = 0$, whenever $s \geq 2$, for all integers $t$ (this fact 
is well-known; as a reference for the argument, see, 
for example, 
\cite[proof of Theorem 2.9]{symondsetal}), and this vanishing 
result implies that the spectral sequence is strongly convergent, 
by \cite[Lemma 5.48]{Thomason}. 
\end{proof}

\section{Filtered diagrams of suitably finite triples and their colimits}\label{consequences}

In this section, we extend Definitions 
\ref{goodpair} and \ref{def:homotopy} to the case of a filtered 
diagram of $G$--spectra.

\begin{Def}\label{hyperdef}
Let $G$ be a profinite group with $\mathcal{N}$ a fixed inverse system of open normal subgroups of $G$, and let 
$\{X_\mu\}_{\mu}$ be a filtered diagram of 
$G$--spectra (thus, the morphisms in the diagram are $G$--equivariant), 
such that for each $\mu$, $(G, X_\mu, \mathcal{N})$ is a suitably 
finite triple and $X_\mu$ is a fibrant spectrum. 
We refer to 
$(G, \{X_\mu\}_\mu, \mathcal{N})$ as a {\em 
suitably filtered triple}. 
\end{Def}

Let $(G, \{X_\mu\}_\mu, \mathcal{N})$ be a suitably filtered triple. 
Since the colimit 
of a filtered 
diagram of weak equivalences between fibrant spectra is a weak 
equivalence, there is a zigzag of 
$G$--equivariant maps 
\begin{equation}\label{zigzag}\zig
\colim_\mu X_\mu \overset{\simeq}{\longrightarrow} 
\colim_\mu \holim_\Delta \mathrm{Sets}(G^{\bullet+1}, (X_\mu)_f) 
\overset{\simeq}{\longleftarrow} 
\colim_\mu \mspace{1mu}(X_\mu)^\mathrm{dis}_\mathcal{N}\end{equation}
that are weak equivalences in $\Sigma\mathrm{Sp}$ (since each 
$\mathrm{Sets}(G^{\bullet+1}, (X_\mu)_f)$ is a cosimplicial fibrant spectrum, 
each $\holim_\Delta \mathrm{Sets}(G^{\bullet+1}, (X_\mu)_f)$ is a fibrant spectrum; 
also, by the proof of Theorem \ref{cool}, 
each $(X_\mu)^\mathrm{dis}_\mathcal{N}$ is a fibrant discrete $G$--spectrum, and thus, 
by \cite[Corollary 5.3.3]{joint}, each $(X_\mu)^\mathrm{dis}_\mathcal{N}$ 
is a fibrant spectrum). The right end of zigzag (\ref{zigzag}) satisfies 
\[\colim_\mu\mspace{1mu}(X_\mu)^\mathrm{dis}_\mathcal{N} = 
\colim_\mu \colim_{\alpha \in \Lambda} \holim_\Delta 
\mathrm{Sets}(G^{\bullet+1}, (X_\mu)_f)^{N_\alpha}\] and 
$\colim_\mu (X_\mu)^\mathrm{dis}_\mathcal{N}$ is a 
discrete $G$--spectrum. (In zigzag (\ref{zigzag}), 
since each $X_\mu$ is a fibrant spectrum, 
the fibrant replacement in each $(X_\mu)_f$ 
is not necessary. However, we believe that by leaving the 
$(-)_f$ in each $(X_\mu)_f$ and by continuing to use the 
maps $i_{X_\mu}$ as previously 
defined (in the proof of Lemma \ref{lemma}), 
our presentation is less cumbersome.) Notice that for every 
integer $t$, our hypotheses on the triple and 
zigzag (\ref{zigzag}) imply that the composition
\begin{equation}\label{3and4}\zig
\colim_\mu \pi_t(X_\mu) \xrightarrow{\,\cong\,} \pi_t(\colim_\mu X_\mu) \xrightarrow{\, \cong \,} 
\pi_t(\colim_\mu (X_\mu)^\mathrm{dis}_\mathcal{N}) \xrightarrow{\,\cong\,} 
\colim_\mu \pi_t((X_\mu)^\mathrm{dis}_\mathcal{N})\end{equation} consists of three 
isomorphisms in the category of 
discrete $G$--modules (in particular, 
each of the four abelian groups above is a discrete $G$--module). 

\begin{Def}\label{naturalforcolimit}
Given a suitably filtered triple $(G, \{X_\mu\}_\mu, \mathcal{N})$, 
the weak equivalences in zigzag (\ref{zigzag}) imply that 
the $G$--spectrum $\colim_\mu X_\mu$ can be 
identified with the discrete $G$--spectrum 
$\colim_\mu (X_\mu)^\mathrm{dis}_\mathcal{N}$. Thus, 
it is natural to define 
\[(\colim_\mu X_\mu)^{hG} = \bigl(\colim_\mu 
\mspace{1mu}(X_\mu)^\mathrm{dis}_\mathcal{N}\bigr)^{\negthinspace hG}.\] We can 
extend this definition to an arbitrary closed subgroup $K$ in $G$: since 
the $K$--spectrum $\colim_\mu X_\mu$ can 
be regarded as the discrete $K$--spectrum $\colim_\mu (X_\mu)^\mathrm{dis}_\mathcal{N}$, we define
\[(\colim_\mu X_\mu)^{hK} = \bigl(\colim_\mu 
\mspace{1mu}(X_\mu)^\mathrm{dis}_\mathcal{N}\bigr)^{\negthinspace hK}.\]
\end{Def}

\begin{Rk}\label{finite}
Let $(G, \{X_\mu\}_\mu, \mathcal{N})$ be a suitably 
filtered triple and let $K$ be a closed subgroup of $G$. Suppose that $K$ is finite, so that its topology is both 
profinite and discrete. It follows that any $K$--spectrum can itself be regarded 
as a discrete $K$--spectrum, whenever desired. Thus, the notation $(\colim_\mu X_\mu)^{hK}$ 
can mean $\bigl((\colim_\mu X_\mu)_{fK}\bigr)^{\mspace{-2mu} K}$ or it can mean $\bigl(\colim_\mu (X_\mu)^\mathrm{dis}_\mathcal{N}\bigr)^{\mspace{-2mu}hK}$. In the remainder of this remark, to avoid any ambiguity, 
we take $(\colim_\mu X_\mu)^{hK}$ to have the latter meaning, $\bigl(\colim_\mu (X_\mu)^\mathrm{dis}_\mathcal{N}\bigr)^{\mspace{-2mu}hK}$, and 
for the former meaning, $\bigl((\colim_\mu X_\mu)_{fK}\bigr)^{\mspace{-2mu} K}$, 
we just write it out as needed. Since (\ref{zigzag}) can be regarded as 
a zigzag of 
weak equivalences in the category of discrete $K$--spectra, 
there is a zigzag of weak equivalences
\[\bigl((\colim_\mu X_\mu)_{fK}\bigr)^{\mspace{-2mu}K} \xrightarrow{\,\simeq\,} 
\bigl(\colim_\mu \holim_\Delta \mathrm{Sets}(G^{\bullet+1}, 
(X_\mu)_f)\bigr)^{\mspace{-2mu} hK} 
\xleftarrow{\,\simeq\,} \bigl(\colim_\mu 
(X_\mu)^\mathrm{dis}_\mathcal{N}\bigr)^{\mspace{-2mu} hK}.\] 
Also, 
given an arbitrary $K$--spectrum $Y$, let $Y \to Y_\mathrm{fib}$ be a trivial cofibration to a fibrant object, in $K\text{-}\Sigma\mathrm{Sp}$, the category of $K$--spectra. Then we have 
\[Y^{\widetilde{h}K} = (Y_\mathrm{fib})^K 
\simeq (Y_{fK})^K,\] where 
the last equivalence follows from the fact that 
$Y_\mathrm{fib}$ is fibrant in 
$\Sigma\mathrm{Sp}_K$ (and this fibrancy assertion is true because 
the functor $(-)_d \: K\text{-}\Sigma\mathrm{Sp} \to \Sigma\mathrm{Sp}_K$ 
preserves fibrant objects and 
$(Y_\mathrm{fib})_d \cong Y_\mathrm{fib}$ is an isomorphism in 
$\Sigma\mathrm{Sp}_K$). We conclude that when $K$ is 
finite, there are equivalences
\[(\colim_\mu X_\mu)^{hK} \simeq 
\bigl((\colim_\mu X_\mu)_{fK}\bigr)^{\mspace{-2mu}K} \simeq 
(\colim_\mu X_\mu)^{\widetilde{h}K},\] as one would expect.
\end{Rk}

\par
We say that a profinite group $G$ has {\em finite virtual cohomological 
dimension} (``finite v.c.d.") if $G$ contains an open subgroup that has 
finite c.d. Under the assumption that $G$ has this property, the following 
result gives a descent spectral sequence for the 
situation described by Definition \ref{hyperdef}.  

\begin{Thm}\label{dss}
Let $G$ be a profinite group with finite v.c.d. 
If $(G, \{X_\mu\}_\mu, \mathcal{N})$ is a suitably filtered 
triple and $K$ is a closed subgroup of $G$, 
then there is a conditionally convergent descent spectral 
sequence $E_r^{\ast, \ast}\mspace{-1.6mu}(K)$ that has the form
\begin{equation}\zig\label{dssform}
E_2^{s,t}\mspace{-1.6mu}(K) = H^s_c(K; \pi_t(\colim_\mu X_\mu)) 
\Longrightarrow \pi_{t-s}\bigl((\colim_\mu X_\mu)^{hK}\bigr).
\end{equation}
\end{Thm}
\begin{Rk}
If $G$ has a good filtration, then condition (d) of 
Definition \ref{goodfiltration} implies that $G$ has finite v.c.d. Thus, 
if $(G, \{X_\mu\}_\mu, \mathcal{N})$ is a suitably filtered triple such that 
there is some $\mu_0 \in \{\mu\}_\mu$ for which 
the triple $(G, X_{\mu_0}, \mathcal{N})$ 
satisfies the hypotheses of Theorem \ref{mainresult}, then 
$G$ has finite v.c.d. and the first sentence of Theorem \ref{dss} can be 
omitted.
\end{Rk}
\begin{proof}[Proof of Theorem \ref{dss}]
Let $U$ be an open subgroup of $G$ that 
has finite c.d. Then $U \cap K$ is an open subgroup of $K$, and since 
$U$ has finite c.d. and $U \cap K$ is closed in $U$, 
there exists some $r$ such that for any discrete 
$(U \cap K)$--module $M$, 
\[H^s_c(U \cap K; M) \cong H^s_c(U; \mathrm{Coind}^U_{U \cap K}\mspace{-1mu}(M)) = 0, \ \ \ \text{whenever} \ s>r,\] by Shapiro's Lemma. This shows that 
$K$ has finite v.c.d. Therefore, \cite[proofs of Theorem 3.2.1, Proposition 3.5.3]{joint} and 
\cite[proof of Theorem 7.9]{cts} yield the conditionally convergent 
spectral sequence 
\[E_2^{s,t} = H^s_c(K; \pi_t(\colim_\mu \mspace{1mu}(X_\mu)^\mathrm{dis}_\mathcal{N})) 
\Longrightarrow \pi_{t-s}\Bigl(\bigl(\colim_\mu 
\mspace{1mu}(X_\mu)^\mathrm{dis}_\mathcal{N}\bigr)^{\negthinspace hK}\Bigr),\]
and this is the desired spectral sequence, since the middle map in composition 
(\ref{3and4}) is an isomorphism of discrete $K$--modules. 

We provide some more detail (based on the above two references) because it will 
be useful to us later. Since $K$ has finite v.c.d., 
\[\bigl(\colim_\mu 
\mspace{1mu}(X_\mu)^\mathrm{dis}_\mathcal{N}\bigr)^{\negthinspace hK} \simeq 
\holim_\Delta \Gamma_K^\bullet \colim_\mu 
\mspace{1mu}(X_\mu)^\mathrm{dis}_\mathcal{N},\] and for each $m \geq 0$, 
the $m$-cosimplices of 
cosimplicial spectrum $\Gamma_K^\bullet \colim_\mu 
\mspace{1mu}(X_\mu)^\mathrm{dis}_\mathcal{N}$ satisfy the isomorphism
\begin{equation}\zig\label{uselater}
\bigl(\Gamma_K^\bullet \colim_\mu 
\mspace{1mu}(X_\mu)^\mathrm{dis}_\mathcal{N}\bigr)^{\negthinspace m} 
\cong \colim_{V \vartriangleleft_o K^m} \textstyle{\prod}_{K^m/V} 
\mspace{1mu}\displaystyle{\colim_\mu} \mspace{1mu}(X_\mu)^\mathrm{dis}_\mathcal{N},
\end{equation} where $K^m$ is the $m$-fold Cartesian product of $K$ ($K^0$ is the trivial group $\{e\}$, equipped with the discrete topology). (For more 
detail about this, we refer the reader to \cite[Sections 2.4, 3.2]{joint}.) 

The above spectral sequence is the homotopy spectral sequence for 
the spectrum 
$\holim_\Delta \Gamma_K^\bullet \colim_\mu 
(X_\mu)^\mathrm{dis}_\mathcal{N}$. 
Based on \cite[proof of Theorem 3.2.1]{joint} and 
\cite[proof of Theorem 7.9]{cts}, the reader might expect us to instead form 
the homotopy spectral sequence for 
$\holim_\Delta \Gamma_K^\bullet \bigl(\colim_\mu 
(X_\mu)^\mathrm{dis}_\mathcal{N}\bigr)_{\negthinspace fK}$. But since each 
$(X_\mu)^\mathrm{dis}_\mathcal{N}$ is a fibrant spectrum, 
$\colim_\mu (X_\mu)^\mathrm{dis}_\mathcal{N}$ is already a fibrant spectrum, so 
that we do not need to apply $(-)_{fK}$ to it (so that we are taking the 
homotopy limit of a cosimplicial fibrant spectrum). 
\end{proof}

Notice that if $(G, \{X_\mu\}_\mu, \mathcal{N})$ is a suitably filtered triple, 
then for each $\mu' \in \{\mu\}_\mu$, $(G, \{X_\mu\}_{\mu \in \{\mu'\}}, \mathcal{N})$ 
is a suitably filtered triple, so that Definition \ref{naturalforcolimit} gives 
\[(X_{\mu'})^{hK} = \bigl((X_{\mu'})^\mathrm{dis}_\mathcal{N}\bigr)^{\negthinspace 
hK},\] for any closed subgroup $K$ of $G$.

\begin{Thm}\label{commute}
Let $G$ be a profinite group with finite v.c.d., let $(G, \{X_\mu\}_\mu, \mathcal{N})$ be a suitably filtered triple such that 
$\{\mu\}_\mu$ is a directed poset, and let $K$ be a closed subgroup of $G$.  
If there exists 
a nonnegative integer $r$ such that for all $t \in \mathbb{Z}$ and each $\mu$, $H^s_c(K; \pi_t(X_\mu)) = 0$ 
whenever $s > r$, then descent spectral sequence 
$E_r^{\ast, \ast}\mspace{-1.6mu}(K)$ in 
$\mathrm{(}\mspace{0mu}$\ref{dssform}$\mspace{2mu}\mathrm{)}$ is strongly convergent and 
there is an equivalence of spectra
\[(\colim_\mu X_\mu)^{hK} \simeq \colim_\mu \mspace{1mu}(X_\mu)^{hK}.\]
\end{Thm}
\begin{proof}
For all $t \in \mathbb{Z}$, when $s > r$, we have
\[E_2^{s,t}(K) = H^s_c(K; \pi_t(\colim_\mu X_\mu)) \cong \colim_\mu 
H^s_c(K; \pi_t(X_\mu)) = 0,\] so that the spectral sequence is strongly 
convergent, by \cite[Lemma 5.48]{Thomason}.

If $V$ is an open normal subgroup of $K^m$, where $m \geq 0$, then 
$K^m/V$ is finite, and hence, isomorphism (\ref{uselater}) implies that 
\[\bigl(\Gamma_K^\bullet \colim_\mu 
\mspace{1mu}(X_\mu)^\mathrm{dis}_\mathcal{N}\bigr)^{\negthinspace m} 
\cong \colim_\mu \colim_{V \vartriangleleft_o K^m} \textstyle{\prod}_{K^m/V} 
\mspace{1mu}(X_\mu)^\mathrm{dis}_\mathcal{N} 
\cong \displaystyle{\colim_\mu} \bigl(\Gamma_K^\bullet (X_\mu)^\mathrm{dis}_\mathcal{N}\bigr)^{\negthinspace m},\] so that there is an isomorphism 
\[\Gamma_K^\bullet \colim_\mu 
\mspace{1mu}(X_\mu)^\mathrm{dis}_\mathcal{N} \cong 
\displaystyle{\colim_\mu} 
\,\Gamma_K^\bullet (X_\mu)^\mathrm{dis}_\mathcal{N}\] of cosimplicial spectra. 
Therefore, we have
\[
\bigl(\colim_\mu 
\mspace{1mu}(X_\mu)^\mathrm{dis}_\mathcal{N}\bigr)^{\negthinspace hK} 
\simeq 
\holim_\Delta \Gamma_K^\bullet \colim_\mu 
\mspace{1mu}(X_\mu)^\mathrm{dis}_\mathcal{N} \cong 
\holim_\Delta \colim_\mu \Gamma_K^\bullet (X_\mu)^\mathrm{dis}_\mathcal{N},\]
which gives
\begin{align*}
(\colim_\mu X_\mu)^{hK} \simeq \holim_\Delta \colim_\mu \Gamma_K^\bullet (X_\mu)^\mathrm{dis}_\mathcal{N} \longleftarrow 
& \colim_\mu \holim_\Delta \Gamma_K^\bullet (X_\mu)^\mathrm{dis}_\mathcal{N} 
\\ & 
\simeq \colim_\mu \bigl((X_\mu)^\mathrm{dis}_\mathcal{N}\bigr)^{\negthinspace hK} 
= \colim_\mu (X_\mu)^{hK},
\end{align*} 
and the canonical colim/holim exchange map 
above is a weak equivalence if there exists a nonnegative integer $r$ such that 
for every $t$ and all $\mu$, 
\[H^s\bigl[ \pi_t\bigl(\Gamma^\ast_K (X_\mu)^\mathrm{dis}_\mathcal{N}\bigr)\bigr] = 0, \ \ \ \text{when} \ s > r,\] 
by \cite[Proposition 3.4]{Mitchell}. The proof is completed by noting that there 
are isomorphisms
\[H^s\bigl[ \pi_t\bigl(\Gamma^\ast_K (X_\mu)^\mathrm{dis}_\mathcal{N}\bigr)\bigr] 
\cong H^s_c(K; \pi_t((X_\mu)^\mathrm{dis}_\mathcal{N})) \cong 
H^s_c(K; \pi_t(X_\mu)),\] for all $s \geq 0$.
\end{proof}

\section{The proofs of Theorems \ref{appliedDSS} and 
\ref{theoremaboutmap}}\label{section-proof}

After proving Theorem \ref{appliedDSS}, a task which 
ends with (\ref{finally}), 
we prove Theorem \ref{theoremaboutmap}.

Let $p \geq 5$ and let $K$ be any closed subgroup of $\mathbb{Z}_p^\times$. 
As noted in the proof of Theorem \ref{illustrate}, $\mathbb{Z}_p$ 
has finite c.d., and since it is open in 
$\mathbb{Z}_p^\times$, $\mathbb{Z}_p^\times$ has finite v.c.d. Also, 
in the introduction (see Remark \ref{terminology}), we showed that 
\[\bigl(\mathbb{Z}_p^\times, \bigl\{\bigl(K(KU_p) \wedge \Sigma^{-jd}V(1)\bigr)_{\negthinspace f}\bigr\}_{\negthinspace j \geq 0}, \mathcal{N}\bigr),\]
where $\mathcal{N}$ is as defined in Remark \ref{terminology}, is a suitably 
filtered triple. Therefore, by Theorem \ref{dss}, there is a conditionally convergent 
descent spectral sequence that has the form 
\begin{equation}\label{8dss}\zig
E_2^{s,t} \Rightarrow 
\pi_{t-s}\bigl(\bigl(K(KU_p) \wedge v_2^{-1}V(1)\bigr)^{\negthinspace hK}\bigr),\end{equation} 
where
\begin{align*}
E_2^{s,t} & = H^s_c\bigl(K; \pi_t\bigl(\colim_{j \geq 0}\bigl(K(KU_p) \wedge 
\Sigma^{-jd}V(1)\bigr)_{\negthinspace f}\bigr)\bigr) \\
& \cong H^s_c(K; \pi_t(K(KU_p) \wedge V(1))[v_2^{-1}]),\end{align*} as 
desired.

Since $p \geq 5$, $V(1)$ is a homotopy commutative and homotopy associative ring spectrum \cite{oka}, 
so that $\pi_\ast(K(KU_p) \wedge V(1))$ is a graded right 
$\pi_\ast(V(1))$--module, and 
hence, $\pi_\ast(K(KU_p) \wedge V(1))$ is a unitary $\mathbb{F}_p$--module. 
It follows that for every integer $t$, the finite abelian group 
$\pi_t(K(KU_p) \wedge V(1))$ is a $p$--torsion group (that is, $pm = 0$, 
for all $m \in \pi_t(K(KU_p) \wedge V(1))$). 

Given any profinite group $G$, we use $\mathrm{cd}_p(G)$ to denote its 
cohomological $p$--dimension. Since $K$ is closed in $\mathbb{Z}_p^\times$, 
\[\mathrm{cd}_p(K) \leq \mathrm{cd}_p(\mathbb{Z}_p^\times) = \mathrm{cd}_p(\mathbb{Z}_p) = 1,\] where the first equality is due to the fact that 
$\mathbb{Z}_p$ is the $p$--Sylow subgroup of $\mathbb{Z}_p^\times$, and hence, 
\begin{equation}\zig\label{vanishingreplacement}
H^s_c(K; M) = 0, \ \ \ \text{for all} \ s \geq 2,
\end{equation}
whenever $M$ is a discrete $K$--module that is also $p$--torsion. Now 
choose any $j \geq 0$. For each $t \in \mathbb{Z}$ and all $s \geq 0$, 
there is an isomorphism 
\[
H^s_c\bigl(K ; \pi_t\bigl(\bigl(K(KU_p) \wedge \Sigma^{-jd}V(1)\bigr)_{\negthinspace f}\bigr)\bigr) 
\cong H^s_c(K; \pi_{t+jd}(K(KU_p) \wedge V(1))).\] Then (\ref{vanishingreplacement}) implies that for every integer $t$, 
\begin{equation}\zig\label{vanishngtwo}
H^s_c\bigl(K ; \pi_t\bigl(\bigl(K(KU_p) \wedge \Sigma^{-jd}V(1)\bigr)_{\negthinspace f}\bigr)\bigr) = 0, \ \ \ \text{for all} \ s \geq 2, 
\end{equation}
since the discrete $K$--module $\pi_{t+jd}(K(KU_p) \wedge V(1))$ is $p$--torsion.

We have now verified the hypotheses of Theorem \ref{commute}, so that 
descent spectral sequence (\ref{8dss}) is strongly convergent, $E_2^{s,t} = 0$ 
for all integers $t$ whenever $s \geq 2$ (see the first sentence of the proof of 
Theorem \ref{commute}), and there is the equivalence
\begin{equation}\zig\label{desiredequivone}
\bigl(K(KU_p) \wedge v_2^{-1}V(1)\bigr)^{\negthinspace hK} 
\simeq \colim_{j \geq 0} \bigl(\bigl(K(KU_p) \wedge \Sigma^{-jd}V(1)\bigr)_{\negthinspace f}\bigr)^{\mspace{-2mu}hK}.\end{equation}

Let $G$ be any profinite group and let $X_1$ and $X_2$ be arbitrary $G$--spectra, 
such that $(G, X_1, \mathcal{U})$ and $(G, X_2, \mathcal{U})$ are suitably finite 
triples (the inverse system $\mathcal{U}$ is the same in each triple) and there is a weak equivalence 
$w \: X_1 \to X_2$ in $\gspt$. The equivalence $w$ induces the commutative diagram 
\[
\xymatrix{X_1 \ar^-\simeq[r] \ar^-\simeq_-w[d] & 
\smash{\displaystyle{\holim_\Delta \mathrm{Sets}}}(G^{\bullet+1}, 
(X_1)_f) \ar[d] & (X_1)^\mathrm{dis}_\mathcal{U} \ar_-\simeq[l] \ar^-{w^\mathrm{dis}_\mathcal{U}}[d] \\
X_2 \ar^-\simeq[r] & \displaystyle{\holim_\Delta \mathrm{Sets}(G^{\bullet+1}, 
(X_2)_f)} & (X_2)^\mathrm{dis}_\mathcal{U} \ar_-\simeq[l]}\] in which 
each ``$\simeq$" denotes a weak equivalence in $\gspt$. From the left commutative square, it follows that the middle vertical map in the diagram is a weak equivalence in $\gspt$, and hence, the right commutative square implies that the 
$G$--equivariant map $w^\mathrm{dis}_\mathcal{U}$ is a weak equivalence of spectra, which allows us to conclude that $w^\mathrm{dis}_\mathcal{U}$ is a 
weak equivalence in $\Sigma\mathrm{Sp}_G$. 

As in Definition \ref{naturalforcolimit}, for any suitably finite triple $(G, X, \mathcal{U})$ and any closed subgroup $P$ of $G$, it is natural to define 
\[X^{hP} = (X^\mathrm{dis}_\mathcal{U})^{hP}\] (this extends Definition \ref{def:homotopy}). For any $P$, since $w^\mathrm{dis}_\mathcal{U}$ is a weak equivalence 
in the category of discrete $P$--spectra, it follows that the induced map 
\[(X_1)^{hP} = ((X_1)^\mathrm{dis}_\mathcal{U})^{hP} \xrightarrow{\,\simeq\,} 
((X_2)^\mathrm{dis}_\mathcal{U})^{hP} = (X_2)^{hP}\] is a weak equivalence. 

For each $j \geq 0$, the triples \[(\mathbb{Z}_p^\times, K(KU_p) \wedge \Sigma^{-jd}V(1), \mathcal{N}) \ \ \ \text{and} \ \ \ \bigl(\mathbb{Z}_p^\times, \bigl(K(KU_p) \wedge \Sigma^{-jd}V(1)\bigr)_{\negthinspace f}, \mathcal{N}\bigr)\] are suitably finite, the 
natural fibrant replacement map \[K(KU_p) \wedge \Sigma^{-jd}V(1) \xrightarrow{\,\simeq\,} 
\bigl(K(KU_p) \wedge \Sigma^{-jd}V(1)\bigr)_{\negthinspace f}\] is a weak 
equivalence in the category of $\mathbb{Z}_p^\times$--spectra, 
and $K(KU_p) \wedge \Sigma^{-jd}V(1)$ can be identified with the discrete 
$\mathbb{Z}_p^\times$--spectrum $(K(KU_p) \wedge \Sigma^{-jd}V(1))^\mathrm{dis}_\mathcal{N}$, as in Definition \ref{def:homotopy}. It follows from the 
above discussion that for each $j \geq 0$ and each closed subgroup $K$, 
there is the definition 
\[(K(KU_p) \wedge \Sigma^{-jd}V(1))^{hK} = 
\bigl((K(KU_p) \wedge \Sigma^{-jd}V(1))^\mathrm{dis}_\mathcal{N}\bigr)^{\mspace{-2mu} hK}\] and 
there is a weak equivalence 
\[(K(KU_p) \wedge \Sigma^{-jd}V(1))^{hK} \xrightarrow{\,\simeq\,} 
\bigl(\bigl(K(KU_p) \wedge \Sigma^{-jd}V(1)\bigr)_{\negthinspace f}\bigr)^{\mspace{-2mu} hK}\] 
between fibrant spectra, giving a weak equivalence 
\begin{equation}\zig\label{desiredequivtwo} 
\colim_{j \geq 0} (K(KU_p) \wedge \Sigma^{-jd}V(1))^{hK} 
\xrightarrow{\,\simeq\,} \colim_{j \geq 0} 
\bigl(\bigl(K(KU_p) \wedge \Sigma^{-jd}V(1)\bigr)_{\negthinspace f}\bigr)^{\mspace{-2mu} hK}.\end{equation}
From (\ref{desiredequivone}) and (\ref{desiredequivtwo}), we obtain an 
equivalence 
\begin{equation}\label{finally}\zig
\bigl(K(KU_p) \wedge v_2^{-1}V(1)\bigr)^{\mspace{-2mu} hK} \simeq 
\colim_{j \geq 0} (K(KU_p) \wedge \Sigma^{-jd}V(1))^{hK}.
\end{equation}

\begin{proof}[Proof of Theorem \ref{theoremaboutmap}]
Setting $n=1$ in (\ref{galois-extension}) gives 
the $K(1)$--local profinite $\mathbb{Z}_p^\times$--Galois extension $L_{K(1)}(S^0) \to KU_p$, and this map yields a $\mathbb{Z}_p^\times$--equivariant map 
$K(L_{K(1)}(S^0)) \to K(KU_p)$, with $\mathbb{Z}_p^\times$ acting trivially on 
$K(L_{K(1)}(S^0))$. Thus, for each $j \geq 0$, the induced map 
\[K(L_{K(1)}(S^0)) \wedge \Sigma^{-jd} V(1) \to K(KU_p) \wedge \Sigma^{-jd} V(1) 
\xrightarrow{\,\simeq\,} \bigl(K(KU_p) \wedge \Sigma^{-jd} V(1)\bigr)_{\negthinspace f}\] 
is $\mathbb{Z}_p^\times$--equivariant, giving the canonical map to the fixed points, 
\begin{equation}\label{pre-map}\zig
K(L_{K(1)}(S^0)) \wedge \Sigma^{-jd} V(1) \to 
\bigl(\bigl(K(KU_p) \wedge \Sigma^{-jd} V(1)\bigr)_{\negthinspace f}\bigr)^{\mspace{-1mu}\mathbb{Z}_p^\times}.\end{equation} It follows that 
there is the map
\begin{equation}\zig\label{mapeins}
K(L_{K(1)}(S^0)) \wedge v_2^{-1}V(1) \to 
\colim_{j \geq 0} \bigl(\bigl(K(KU_p) \wedge \Sigma^{-jd} V(1)\bigr)_{\negthinspace f}\bigr)^{\mspace{-1mu}\mathbb{Z}_p^\times},\end{equation} which is 
defined to be the composition
\[K(L_{K(1)}(S^0)) \wedge v_2^{-1}V(1) 
\xrightarrow{\,\cong\,} \colim_{j \geq 0} (K(L_{K(1)}(S^0)) \wedge \Sigma^{-jd} V(1)) 
\rightarrow \colim_{j \geq 0} (KV_j)^{\mathbb{Z}_p^\times},\] where 
here and below, we use the notation 
\[KV_j := \bigl(K(KU_p) \wedge \Sigma^{-jd} V(1)\bigr)_{\negthinspace f} \, , \ \ \ 
\text{for} \ j \geq 0,\] to keep certain expressions from being too long (and the second 
map in the composition is obtained by taking a colimit of the maps given by (\ref{pre-map})).

For the diagram of $\mathbb{Z}_p^\times$--equivariant maps 
\[\bigl\{i_{\scriptscriptstyle{KV_j}} \: KV_j \xrightarrow{\,\simeq\,} \holim_\Delta \mathrm{Sets}((\mathbb{Z}_p^\times)^{\bullet+1}, (KV_j)_f)\bigr\}_{\negthinspace j \geq 0},\] taking fixed points and then the colimit gives the 
canonical map 
\begin{equation}\zig\label{mapzwei}
\colim_{j \geq 0} (KV_j)^{\mathbb{Z}_p^\times} \to \colim_{j \geq 0} \bigl(\holim_\Delta \mathrm{Sets}((\mathbb{Z}_p^\times)^{\bullet+1}, (KV_j)_f)\bigr)^{\mspace{-2mu}\mathbb{Z}_p^\times}.\end{equation}
Also, for each $j \geq 0$ (and with $\mathcal{N}$ as defined in Remark \ref{terminology}), there are natural isomorphisms 
\begin{align*}
\bigl(\holim_\Delta \mathrm{Sets}((\mathbb{Z}_p^\times)^{\bullet+1}, (KV_j)_f)\bigr)^{\mspace{-2mu}\mathbb{Z}_p^\times} 
& \cong \Bigl(\colim_{\mspace{1mu}N \vartriangleleft_o \mathbb{Z}_p^\times} \bigl(\holim_\Delta \mathrm{Sets}((\mathbb{Z}_p^\times)^{\bullet+1}, (KV_j)_f)\bigr)^{\mspace{-2mu} N}\Bigr)^{\mspace{-2mu}\mathbb{Z}_p^\times} \\ & \cong 
\bigl((KV_j)^\mathrm{dis}_\mathcal{N}\bigr)^{\mspace{-2mu}\mathbb{Z}_p^\times},
\end{align*} 
where the last step is due to the isomorphism 
\[\colim_{N \vartriangleleft_o \mathbb{Z}_p^\times} \bigl(\holim_\Delta \mathrm{Sets}((\mathbb{Z}_p^\times)^{\bullet+1}, (KV_j)_f)\bigr)^{\mspace{-2mu} N}
\cong (KV_j)^\mathrm{dis}_\mathcal{N}\] in the category of 
discrete $\mathbb{Z}_p^\times$--spectra (which itself is valid by Remark \ref{yieldsgood} (see its first two sentences)), and hence, there is the isomorphism
\begin{equation}\zig\label{mapdrei}
\colim_{j \geq 0} \bigl(\holim_\Delta \mathrm{Sets}((\mathbb{Z}_p^\times)^{\bullet+1}, (KV_j)_f)\bigr)^{\mspace{-2mu}\mathbb{Z}_p^\times} \xrightarrow{\,\cong\,} 
\colim_{j \geq 0} \bigl((KV_j)^\mathrm{dis}_\mathcal{N}\bigr)^{\mspace{-2mu} 
\mathbb{Z}_p^\times}.
\end{equation}

Finally, there is the composition of canonical maps
\begin{equation}\zig\label{mapvier}
\colim_{j \geq 0} \bigl((KV_j)^\mathrm{dis}_\mathcal{N}\bigr)^{\mspace{-2mu} 
\mathbb{Z}_p^\times} 
\to \bigl(\colim_{j \geq 0} (KV_j)^\mathrm{dis}_\mathcal{N}\bigr)^{\mspace{-2mu} \mathbb{Z}_p^\times} = (C^\mathrm{dis}_p)^{\mathbb{Z}_p^\times} \to 
\bigl((C^\mathrm{dis}_p)_{f\mathbb{Z}_p^\times}\bigr)^{\mspace{-2mu}\mathbb{Z}_p^\times},\end{equation} where 
the first map is due to the universal property of the colimit and the second map 
is obtained by applying fixed points to the fibrant replacement map. 
The target of map (\ref{mapvier}) is equal to 
$\bigl(K(KU_p) \wedge v_2^{-1}V(1)\bigr)^{\mspace{-2mu}h\mathbb{Z}_p^\times}$, and the composition 
of maps (\ref{mapeins}), (\ref{mapzwei}), (\ref{mapdrei}), and (\ref{mapvier}) 
(that is, after omitting the source and target from each map, 
the composition $\xrightarrow{(\ref{mapeins})}\,\xrightarrow{(\ref{mapzwei})}\,\xrightarrow{(\ref{mapdrei})}\,\xrightarrow{(\ref{mapvier})}\,$) 
defines the desired map
\[K(L_{K(1)}(S^0)) \wedge v_2^{-1}V(1) \to 
\bigl(K(KU_p) \wedge v_2^{-1}V(1)\bigr)^{\mspace{-2mu}h\mathbb{Z}_p^\times}.
\qedhere\]
\end{proof}

\section{The proof of Theorem \ref{surprising}}\label{section-surprising}
As in the preceding section, we continue with letting $p \geq 5$. By Theorem \ref{appliedDSS} (in particular, see (\ref{finally})), there is an equivalence 
\begin{align*}
\bigl(K(KU_p) \wedge v_2^{-1}V(1)\bigr)^{h\mathbb{Z}_p^\times} 
\simeq \colim_{j \geq 0} (K(KU_p) \wedge \Sigma^{-jd}V(1))^{h\mathbb{Z}_p^\times},
\end{align*} 
and for each $j \geq 0$, 
$(\mathbb{Z}_p^\times, K(KU_p) \wedge \Sigma^{-jd}V(1), 
\mathcal{N})$ (with $\mathcal{N}$ as defined in Remark \ref{terminology}) is a suitably finite triple. Then by the proof 
of Theorem \ref{cool} (the spectrum $(K(KU_p) \wedge \Sigma^{-jd}V(1))^\mathrm{dis}_\mathcal{N}$ is a 
fibrant discrete $\mathbb{Z}_p^\times$--spectrum, for each $j$), 
there are weak equivalences 
\begin{align*}
\colim_{j \geq 0} (K(KU_p) \wedge 
\Sigma^{-jd}V(1))^{h\mathbb{Z}_p^\times} & = 
\colim_{j \geq 0} \Bigl(\bigl((K(KU_p) \wedge \Sigma^{-jd}V(1))^\mathrm{dis}_\mathcal{N}\bigr)_{\mspace{-3mu}f\mathbb{Z}_p^\times}\Bigr)^{\mspace{-3mu}\mathbb{Z}_p^\times} \\ 
& \xleftarrow{\,\simeq\,} 
\colim_{j \geq 0} \bigl((K(KU_p) \wedge \Sigma^{-jd}V(1))^\mathrm{dis}_\mathcal{N}\bigr)^{\mathbb{Z}_p^\times} \\ 
& \xleftarrow{\,\simeq\,} 
\colim_{j \geq 0} (K(KU_p) \wedge 
\Sigma^{-jd}V(1))^{\widetilde{h}\mathbb{Z}_p^\times}. 
\end{align*}

The last weak equivalence above requires a little more justification. 
Let $J$ denote the indexing category $\{j \geq 0 \}$ for the above colimits. For any profinite group $G$, the model structure on 
$\gspt$ is combinatorial, by \cite[Proposition A.2.8.2]{luriebook}, 
and hence, $(\gspt)^J$, the category of $J$--shaped diagrams in 
$\gspt$, can be equipped with a projective model structure (again, 
by \cite[Proposition A.2.8.2]{luriebook}). Thus, we regard 
$(\zpspt)^J$ as having a projective model structure, and we let 
\[\{K(KU_p) \wedge \Sigma^{-jd}V(1)\}_{_{\mspace{-1.5mu}j \geq 0}} 
\xrightarrow{\,\simeq\,} 
\{(K(KU_p) \wedge \Sigma^{-jd}V(1))_\mathrm{pf}\}_{_{\mspace{-1.5mu}j \geq 0}}\] 
be a trivial cofibration to a fibrant object, in $(\zpspt)^J$. Notice that 
by the proof of Theorem \ref{cool}, the morphism 
\[\{K(KU_p) \wedge \Sigma^{-jd}V(1)\}_{_{\mspace{-1.5mu}j \geq 0}} \xrightarrow{\,\simeq\,} 
\bigl\{\holim_\Delta \mathrm{Sets}((\mathbb{Z}_p^\times)^{\bullet+1}, 
(K(KU_p) \wedge \Sigma^{-jd}V(1))_f)\bigr\}_{_{\mspace{-3mu}j \geq 0}}\] is a weak equivalence to a fibrant object, in $(\zpspt)^J$, 
and hence, there is a morphism
\[\{(K(KU_p) \wedge \Sigma^{-jd}V(1))_\mathrm{pf}\}_{_{\mspace{-1.5mu}j \geq 0}} \xrightarrow{\,\simeq\,} 
\bigl\{\holim_\Delta \mathrm{Sets}((\mathbb{Z}_p^\times)^{\bullet+1}, 
(K(KU_p) \wedge \Sigma^{-jd}V(1))_f)\bigr\}_{_{\mspace{-4mu}j \geq 0}}\] that is a weak equivalence 
between fibrant objects, in $(\zpspt)^J$. As in the proof of 
Theorem \ref{cool}, the 
application of the right Quillen functor \[(-)^{\mathbb{Z}_p^\times} \: 
\zpspt \rightarrow \Sigma\mathrm{Sp}\] to the last morphism induces 
compositions
\begin{align*}
((K(KU_p) \wedge \Sigma^{-jd}V(1))_\mathrm{pf})^{\mathbb{Z}_p^\times} & \mspace{-2mu}\xrightarrow{\simeq}\mspace{-2mu} 
\bigl(\holim_\Delta \mathrm{Sets}((\mathbb{Z}_p^\times)^{\bullet+1}, 
\mspace{-1mu}(K(KU_p) \wedge \Sigma^{-jd}V(1))_f)\bigr)^{\mspace{-1mu}\mathbb{Z}_p^\times} \\
& \xrightarrow{\cong}\mspace{-2mu} 
\bigl((K(KU_p) \wedge \Sigma^{-jd}V(1))^\mathrm{dis}_\mathcal{N}\bigr)^{\mspace{-1mu}\mathbb{Z}_p^\times}\end{align*} 
for all $j \geq 0$, with each composition a weak equivalence between fibrant spectra. Taking the 
colimit over $J$ of these weak equivalences yields the weak 
equivalence 
\[\omega \: \colim_{j \geq 0} ((K(KU_p) \wedge \Sigma^{-jd}V(1))_\mathrm{pf})^{\mathbb{Z}_p^\times} \xrightarrow{\,\simeq\,} 
\colim_{j \geq 0} \bigl((K(KU_p) \wedge \Sigma^{-jd}V(1))^\mathrm{dis}_\mathcal{N}\bigr)^{\mspace{-1mu}\mathbb{Z}_p^\times}.
\] By \cite[Remark A.2.8.5]{luriebook}, every projective cofibration is 
an injective cofibration, so that for each $j$, the map
\[K(KU_p) \wedge \Sigma^{-jd}V(1) \xrightarrow{\,\simeq\,} (K(KU_p) \wedge \Sigma^{-jd}V(1))_\mathrm{pf}\] is a trivial cofibration 
to a fibrant object in $\zpspt$. It follows from this 
that the source of weak equivalence $\omega$ satisfies the 
equality
\[\colim_{j \geq 0} ((K(KU_p) \wedge \Sigma^{-jd}V(1))_\mathrm{pf})^{\mathbb{Z}_p^\times} = 
\colim_{j \geq 0} (K(KU_p) \wedge \Sigma^{-jd}V(1))^{\widetilde{h}\mathbb{Z}_p^\times},\] 
and thus, $\omega$ is the weak equivalence 
that we set out in this paragraph to obtain.

Fix $j \geq 0$. Since $V(1)$ is a finite spectrum, 
$\Sigma^{-jd}V(1)$ is too, and hence, there is an equivalence 
\[(K(KU_p) \wedge 
\Sigma^{-jd}V(1))^{\widetilde{h}\mathbb{Z}_p^\times} \simeq 
(K(KU_p))^{\widetilde{h}\mathbb{Z}_p^\times} \wedge 
\Sigma^{-jd}V(1)\] 
(for example, see \cite[Lemma 6.2.6]{Rognes}; the key point here is that a homotopy limit commutes with smashing with a finite spectrum). 
It follows from our last equivalence that
\begin{align*}
\colim_{j \geq 0} (K(KU_p) \wedge 
\Sigma^{-jd}V(1))^{\widetilde{h}\mathbb{Z}_p^\times} 
& \simeq \colim_{j \geq 0} 
\bigl((K(KU_p))^{\widetilde{h}\mathbb{Z}_p^\times} 
\wedge \Sigma^{-jd}V(1)\bigr) \\
& \cong (K(KU_p))^{\widetilde{h}\mathbb{Z}_p^\times} \wedge 
v_2^{-1}V(1).
\end{align*} 
Putting all of the equivalences above together yields 
\[ 
\bigl(K(KU_p) \wedge v_2^{-1}V(1)\bigr)^{h\mathbb{Z}_p^\times} 
\simeq (K(KU_p))^{\widetilde{h}\mathbb{Z}_p^\times} \wedge 
v_2^{-1}V(1),
\qedhere 
\] 
which is the desired equivalence.

\end{document}